\documentclass[11pt,letterpaper]{amsart}
      \usepackage{amsmath}
      \usepackage{}
      \usepackage{amsfonts}
      \usepackage{mathrsfs}
      \usepackage{amsfonts, latexsym, rawfonts, amsmath, amssymb, amsthm}
      \usepackage[plainpages=false]{hyperref}

      \usepackage{fancyhdr}

      \usepackage{pdfsync}

      \usepackage{esint}
      \usepackage{graphics, graphicx}
      \usepackage{tikz}
      \usepackage{appendix}

      \numberwithin{equation}{section}
      \newcommand{\beq}{\begin{equation}}
      \newcommand{\eeq}{\end{equation}}
      \newcommand{\beqs}{\begin{eqnarray*}}
      \newcommand{\eeqs}{\end{eqnarray*}}
      \newcommand{\beqn}{\begin{eqnarray}}
      \newcommand{\eeqn}{\end{eqnarray}}
      \newcommand{\beqa}{\begin{array}}
      \newcommand{\eeqa}{\end{array}}

      \def\lra{\longrightarrow}

      \def\bc{\begin{center}}
      \def\ec{\end{center}}

      \def\begeq{\begin{equation}}
      \def\endeq{\end{equation}}
      \def\and{\quad{\rm and}\quad}

      \let\lra=\longrightarrow

      \def\mapright\#1{\, \smash{\mathop{\lra}\limits^{\#1}}\, }

      \newtheorem{prop}{Proposition}[section]
      \newtheorem{theo}[prop]{Theorem}
      \newtheorem{lem}[prop]{Lemma}
      
      \newtheorem{cor}[prop]{Corollary}
      \newtheorem{rem}[prop]{Remark}
      
      \newtheorem{defi}[prop]{Definition}
      \newtheorem{conj}[prop]{Conjecture}

\begin{document}

       \title{  $4d$   steady  gradient Ricci solitons  with   nonnegative curvature away from a compact set}

     \author{Ziyi  $\text{Zhao}^{\dag}$ and Xiaohua $\text{Zhu}^{\ddag}$}

 \address{BICMR and SMS, Peking
 University, Beijing 100871, China.}
 \email{ 1901110027@pku.edu.cn\\\ xhzhu@math.pku.edu.cn}

 \thanks {$\ddag$ partially supported  by National Key R\&D Program of China  2020YFA0712800 and  NSFC 12271009.}


     \begin{abstract} In the paper,  we  classify    the blow-down solutions for $4d$ noncompact $\kappa$-noncollapsed  steady  gradient   Ricci solitons $(M, g)$ with  $\rm{Km}\geq 0$   and $\rm{Ric}> 0$ away from a compact set of $M$.     As an application,    we  prove:  any  $4d$ noncompact $\kappa$-noncollapsed   steady    gradient Ricci soliton   with   nonnegative sectional curvature must be  the Bryant Ricci  soliton up to scaling if it admits    a sequence of rescaled flows,  which   converges  subsequently to  a  splitting limit flow $ \bar g(t)= h(t) +ds^2$  on $N\times \mathbb R$, when  $N$ is  a compact manifold.
             \end{abstract}

       \date{}

    \maketitle

   % \tableofcontents

    \setcounter{section}{-1}

  \section{Introduction}

   Steady  (gradient)  Ricci soliton,  as  a   singular model  of type \uppercase\expandafter{\romannumeral2} of Ricci flow,  has been extensively  studied.  In dimensions 2,    Hamilton  \cite{Ham-singular, CLN}  proved that  the  cigar solution is  the  only  2d  steady  Ricci soliton  up to scaling.   In dimensions  3,  Perelman conjectured  that the  Bryant  Ricci soliton is the  only  3d $\kappa$-noncollapsed   steady  Ricci soliton  up to scaling \cite{P}. The conjecture has been proved by    Brendle \cite{Bre-3d}.  These  two Ricci solitons  are both    rotationally symmetric.    In  dimension $n\ge 4$,  besides the  Bryant soliton \cite{Bry},   Lai  recently constructed  a family of ${\rm SO}(n-1)$-symmetric  solutions with  positive curvature operator \cite{Yi-flyingwings}. Thus it is interesting to classify    steady Ricci solitons  under suitable conditions of  symmetry and curvature.

  To characterize   the Bryant Ricci  soliton,  Brendle introduced  the following notion \cite{Bre-high}.

   \begin{defi}\label{asymptotically-cylindrical}  A complete  steady  gradient  Ricci soliton $(M^n, g, f)$ of dimension $n$  is  called asymptotically cylindrical if the following holds:

 (i) The scalar curvature satisfies $\frac{C_1}{\rho(x)}\leq R(x)\leq \frac{C_2}{\rho(x)}$ as $\rho(x)>>1$,  where $C_1$, $C_2$ are two positive constants.

  (ii) Let $p_i$ be an arbitrary sequence of marked points going to infinity. Consider the rescaled metrics
  \begin{align}\label{rescaling-flow}g_{p_i}(t)=r_i^{-1}\phi^*_{r_it}(g),
  \end{align}
  where $r_iR(p_i)=1$  and $\phi_{t}$ is a family of transformations generated by $-\nabla f$, the flow $(M,g_{p_i}(t); p_i)$ converges in the Cheeger-Gromov sense to a family of shrinking cylinders $(\mathbb S^{n-1}\times \mathbb{R},~ \bar{g}(t))$, $t\in (0,1)$. The metric  $\bar g(t)$ is given by
  \begin{align}\label{shrinking-cylinder}\bar g(t)= (1-t) g_{\mathbb S^{n-1}(1)} +ds^2,
  \end{align}
  where  $\mathbb S^{n-1}(1)$ is the unit sphere of dimension $(n-1)$  in $\mathbb R^n$.

  \end{defi}

 In \cite{Bre-high},   Brendle proved  that  any  steady   gradient   Ricci soliton   with nonnegative  sectional curvature must be  isometric to the   Bryant  Ricci soliton up to scaling  if it   is asymptotically cylindrical.  Later,  Deng-Zhu found that  the Brendle's result still holds if one of the  two conditions in Definition \ref{asymptotically-cylindrical} is satisfied for  $\kappa$-noncollapsed steady  Ricci solitons with nonnegative  curvature operator \cite {DZ-JEMS, DZ-SCM}. In a very recent paper \cite{ZZ},  the   Brendle's result    has  been  improved by  assuming  the nonnegative  sectional  curvature  away from a compact set of $M$ together  with   $P$-curvature  pinching condition instead of
 nonnegative  sectional curvature on the whole $M$.
  Thus it is a natural question to ask: is  the Brendle's result  true if  the  condition (ii) is satisfied by (at least) one sequence?  In this paper, we give a positive answer  for 4d $\kappa$-noncollapsed steady  Ricci solitons with nonnegative  sectional curvature.

   Let $(M^n, g)$ be a complete noncompact $\kappa$-noncollapsed steady  Ricci soliton  with  curvature operator  $\rm{Rm}\geq 0$ (sectional curvature $\rm{Km}\geq 0$ for $n=4$)  and $\rm{Ric}> 0$  away from a compact set $K$ of $M$.   Let $p_i\rightarrow \infty$ be any sequence in $M$  and  $g_{p_i}(t)$ the rescaled flow of Ricci soliton  $g$ as in  Definition \ref{asymptotically-cylindrical}.  Then   $(M,g_{p_i}(t); p_i)$ converges to  a splitting  flow  in the Cheeger-Gromov sense,
  \begin{align}\label{splitt-solution} \bar g(t)= h(t) +ds^2, ~ {\rm on}~ N\times \mathbb R,
    \end{align}
       where  $h(t)$  ($t\in (-\infty, 0]$)  is an ancient   $\kappa$-solution on an  $(n-1)$-dimensional $N$, see Proposition \ref{Dimension-reduction}.

   The following alternative principle is one of   main results in this paper.

  \begin{theo}\label{main-theorem}
      Let $(M,  g)$ be a $4d$ complete  noncompact $\kappa$-noncollapsed  steady  gradient   Ricci soliton with  $\rm{Km}\geq 0$   and $\rm{Ric}> 0$ away from a compact set $K$ of $M$.   Let $p_i\rightarrow \infty$ be any  sequence in $M$  and  let  $ \bar g(t)= h(t) +ds^2$ be  the   splitting limit flow of    $(M,g_{p_i}(t); p_i)$  as in (\ref{splitt-solution}).  Then either all  $h(t)$ are families  of  $3d$  shrinking  round quotient   spheres, or  all $h(t)$ are   $3d$ noncompact  ancient $\kappa$-solutions.
  \end{theo}

  We note that both cases  can happen in Theorem \ref{main-theorem}  as seen in  following examples.  For any $2n\ge 4$  and each $Z_k$-group,  Appleton  \cite{App} has  constructed  an example of
   noncompact $\kappa$-noncollapsed   steady   gradient  Ricci soliton with  $\rm{Rm}>0$ on $M\setminus K$, where  all split    ancient   $\kappa$-solutions $h(t)$  form   families  of    shrinking  quotient  spheres  of dimension  $(2n-1)$ with $Z_k$-group.
 In each of   Lai's  examples \cite{Yi-flyingwings} of   noncompact $\kappa$-noncollapsed  steady   gradient  Ricci solitons  with  $\rm{Rm}>0$ on $M$,  all split    ancient   $\kappa$-solutions  $h(t)$ are  noncompact.

  We also note that  $3d$ noncompact ancient $\kappa$-solution (non-flat)  has  been recently classified by    Brendle \cite{Bre-3d-noncpt} (also see \cite{BK}).  Namely,   such soliton   is  isometric to  either  a flow of  $2d$ shrinking cylinders, or a  flow of the Bryant Ricci soliton. On the other hand, by  recent results of Chow-Deng-Ma \cite[Theorem 1.3,  Claim 6.4]{CDM},   both  flows  of the Bryant Ricci soliton and   $2d$ shrinking cylinders  will always  arise  in the second case in  Theorem \ref{main-theorem}, since Theorem \ref{main-theorem} in particular  implies  that  $(M,  g)$ satisfies the canonical model condition in   Chow-Deng-Ma's paper. Thus  we can actually classify    the blow-down solutions for $4d$ noncompact $\kappa$-noncollapsed  steady  gradient   Ricci solitons with  $\rm{Km}\geq 0$   and $\rm{Ric}> 0$ away from a compact set $K$ of $M$.

  As an application of Theorem \ref{main-theorem},  we prove

  \begin{cor}\label{coro}
      Let $(M, g)$ be a $4d$ noncompact $\kappa$-noncollapsed  steady  gradient   Ricci soliton
     with   nonnegative sectional curvature.    Suppose  that there exists a sequence of rescaled flows  $(M,g_{p_i}(t); p_i)$ of  $(M,g)$
which   converges  subsequently to  a family of shrinking  quotient cylinders.
       Then $(M, g)$ is isometric to  the $4d$ Bryant  Ricci soliton up to scaling. Moreover,  the result is still true  only  by assuming  the existence of  $3d$  compact split limit flows $(N, h(t))$ as in (\ref{splitt-solution}).

        \end{cor}

   Our proof of Theorem \ref{main-theorem}  depends on a  deep  classification result   for $3d$ compact $\kappa$-solutions  proved by Brendle-Daskalopoulos-Sesum \cite{BDS} (also see Theorem \ref{Brendle}). But we guess  that  analogs  of Theorem \ref{main-theorem} and Corollary \ref{coro} hold  both
   for general  dimension (cf.  Conjecture \ref{high-Th} and Conjecture \ref{high-Coro}).

The paper is organized as follows.  In Section 1, we prove a splitting result   for  any limit flow of rescaled  flows sequence from a  $\kappa$-noncollapsed steady gradient  Ricci soliton  $(M, g)$ with
     %${\rm Ric}>0$ and
       $\rm{Rm}\geq 0$ on  $M\setminus K$, see Proposition  \ref{Dimension-reduction}.   In Section 2,  we first get a decay estimate of curvature and then study   the level set  geometry  of $(M, g)$  by assuming  the existence  of    compact  split ancient  $\kappa$-solution $(N, h(t))$, see Lemma \ref{compact-R-decay},  Proposition \ref{compact-levelset-compact}, etc.    All  results    in this section hold for any dimension. In Section 3, we focus on a $4d$  steady  Ricci soliton  to get a diameter  estimate of $(N, h(0))$ for all   3d  split flows   $(N, h(t))$, see
       Proposition \ref{noncompact-ancient-solution}. Both of  Theorem \ref{main-theorem} and  Corollary \ref{coro}.
will be proved in Section 4.

  \section{A splitting theorem}

  A complete Riemannian metric $g$ on  $M$ is  called a gradient Ricci soliton if there exists a smooth function $f$ (which is called a defining function)  on $M$ such that
  \begin{equation}\label{Def-soliton}
  R_{ij}(g)+\rho g_{ij}=\nabla_{i}\nabla_{j}f,
  \end{equation}
  where $\rho\in \mathbb{R}$ is a constant. The gradient Ricci soliton is called expanding,  steady and shrinking  according to  $\rho >,  = ,  <0$,  respectively.     These three types  of   Ricci solitons  correspond to three different blow-up solutions  of Ricci flow  \cite{Ham-singular}.

   In case of  steady Ricci solitons, we can rewrite (\ref{Def-soliton}) as
    \begin{align}\label{soliton-equation} 2\operatorname{Ric}(g)=L_Xg,
    \end{align}
    where  $L_X $ is the Lie operator along   the gradient vector field  (VF)  $X = \nabla f$  generalized by $f$.
  Let $\{\phi^*_t\}_{t\in(-\infty,\infty)}$ be   a  1-ps  of transformations   generated by $-X$.   Then
  $g(t)=\phi^*_t(g)$ ($t\in(-\infty,\infty))$ is a solution of Ricci flow. Namely,  $g(t)$ satisfies
  \begin{align}\label{ricci-equ}\frac{\partial g}{\partial t}=-2{\rm Ric}(g), ~ g(0)=g.
  \end{align}
  For simplicity,  we call  $g(t)$  the soliton  Ricci flow of  $(M, g)$.

   By (\ref{soliton-equation}), we have
   \begin{align}\label{R-monotonicity}
       \langle\nabla R,\nabla f\rangle=- 2\operatorname{Ric}(\nabla f,\nabla f),
   \end{align}
   where $R$ is the scalar curvature of $g$.
   It follows
   $$R+|\nabla f|^2={\rm Const.}$$
    Since $R$ is alway positive (\cite{Zh, Ch}),  the above equation can be normalized by
   \begin{align}\label{scalar-equ} R+|\nabla f|^2=1.
   \end{align}

 We recall  that  an ancient $\kappa$-solution is a  $\kappa$-noncollapsed  solution of Ricci flow (\ref{ricci-equ})   with  ${\rm R_m}(\cdot, t)\ge 0$ defined for any $t\in (-\infty, T_0]$.
The following result is a version of  Perelman's compactness theorem for  higher dimensional  ancient  $\kappa$-solutions.

 \begin{theo}\label{diemsion-reduction-positive-curved}
     Let $\left(M^n, g_{i}(t) ; p_i\right)$ $(n\ge 3)$  be any  sequence of   $n$-dimensional   ancient  $\kappa$-solutions on a noncompact manifold  $M$ with $R\left(p_i, 0\right)=1$. Then $(M, g_i(t);~$
     $ p_i)$ subsequently converge to a  splitting flow $(N \times \mathbb{R}, \bar {g}(t);  p_\infty)$ in the Cheeger-Gromov sense. Here
      \begin{align}\label{splitting-2} \bar {g}(t)=h(t) + ds^2,
       \end{align}
     and $(N, h(t))$ is an $(n-1)$-dimensional  ancient $\kappa$-solution.
 \end{theo}

 The convergence of  $\left(M, g_i(t) ; p_i\right)$ comes from \cite[Theorem 3.3]{DZ-TAMS}.  The splitting property in
 (\ref{splitting-2}) can also be  obtained  by   Hamilton's argument  \cite[Lemma 22.2]{Ham-singular}  with the  help of Perelman's asymptotic volume ratio estimate for $\kappa$-solutions \cite[Proposition 41.13]{KL}.
 In fact,  for a sequence of  rescaling Ricci flows  arising from a  steady Ricci soliton,   we can improve Theorem \ref{diemsion-reduction-positive-curved}  under a  weaker curvature condition as follows.

 \begin{prop}\label{Dimension-reduction}
     Let $(M^n, g)$  $(n\ge 3)$ be a noncompact $\kappa$-noncollapsed  steady  gradient Ricci soliton with
     %${\rm Ric}>0$ and
       $\rm{Rm}\geq 0$ away from $K$.   Let   $p_i\rightarrow \infty$ and $(M, g_{p_i}(t); p_i)$  a sequence of rescaled  flows with   $ R_{p_i}\left(p_i, 0\right)=1$ as in (\ref{rescaling-flow}) .  Then $(M, g_{p_i}(t); ~$
       $p_i)$ subsequently converge to a  splitting flow $(N \times \mathbb{R}, \bar {g}(t); p_\infty)$ as
      in  Theorem \ref{diemsion-reduction-positive-curved}. Moreover,   for $n=4$, $\rm{Rm}\geq 0$ can be weakened to $\rm{Km}\geq 0$ away from $K$.

 \end{prop}

 \begin{proof} Since $\rm{Km}\geq 0$ on $M\setminus K$, we have  the Harnack estimate by (\ref{R-monotonicity}),
     \begin{align}\label{harnack-esti}\frac{d}{d t} R(x, t) \geq 0, ~{\rm on}~M\setminus K.
       \end{align}
   Then  according to the proof of Theorem \ref{diemsion-reduction-positive-curved} (see Lemma 3.5-3.7  for details there),  for the convergence part in the proposition,   we   only need  to show that the following asymptotic scalar curvature  estimate,
    \begin{align}\label{ASCR}
          { \rm limsup}_{x\rightarrow \infty} R(x)d^2(o,x)=\infty,
       \end{align}
       where $o\in M$ is a fixed point.
    As a consequence,
 the   rescaled flow   $(M,  g_{p_i}(t); p_i)$ has uniformly  local curvature   estimate, and therefore
      $(M,  g_{p_i}(t); p_i)$ subsequently converges  to a limit   ancient  $\kappa$-solution $(M_\infty, \bar g(t); p_\infty)$.
      %The spliting
      %property of  $\bar g(t)$ as in  (\ref{splitting-2}) can be also obtained.

 We note that  (\ref{ASCR}) is true for any  ancient  $\kappa$-solution  by    Perelman's result of asymptotic zero volume ratio \cite{P, KL}  (cf. \cite[Corollary 2.4]{DZ-TAMS}).  In our case, we   only have   $\rm{Rm}\geq 0$   away from $K$. We will use a different argument to prove  (\ref{ASCR})  below.

        On the contrary,  we  suppose  that  (\ref{ASCR})  is not true.   Then there exists a constant $C>0$, such that
          \begin{align}\label{contradiction-1}R(x)\leq \frac{C}{d^2(o,x)}=o( \frac{1}{d(o,x)}).
       \end{align}
       In particular, the scalar curvature   decays to zero  uniformly.  Due to a result in  \cite[Theorem 2.1]{CDM},  we know that there are two
       constants $c_1, c_2>0$ such that
       \begin{align}\label{linear-f}
       c_1\rho(x)\le f(x)\le c_2\rho(x).
       \end{align}
       Thus by   \cite[Theorem 6.1]{DZ-JEMS} with the help of  (\ref{contradiction-1}) and (\ref{linear-f}), we get
        $$R(x)\geq \frac{C_0}{d(o,x)},$$
        for some constant $C_0$.  But this contradicts to  $(\ref{contradiction-1})$.
   Hence  (\ref{ASCR}) is true.

  In the following,  our goal is to   show   that $\bar g(t)$  is of form  (\ref{splitting-2}).
   First we  prove the volume ratio estimate,
       \begin{align}\label{AVR}
           {\rm AVR}(g)=\lim_{r\rightarrow \infty}\frac{ {\rm Vol}(B(p,r))}{r^n}=0.
       \end{align}

   By (\ref{ASCR}), we can use  the Hamilton's argument in  \cite[Lemma 22.2]{Ham-singular} to find    sequences of points  $q_i\rightarrow \infty$ and number  $s_i>0$ such that
   $\frac{s_i}{d(q_i, o)}\rightarrow 0$,
    \begin{align}\label{growth-infty}R(q_i)s_i^2\rightarrow \infty,
    \end{align}
     and
   \begin{align}\label{harnack-R}
   R(x)\leq 2R(q_i), ~ \forall~ x\in B(q_i, s_i).
   \end{align}
   Consider  a sequence of  the rescaled flows  $(M, g_{q_i}(t);  q_i)$, $t\in(-s_i,0]$,   such that $ R_{q_i}\left(q_i, 0\right)=1$, where $ R_{q_i}(\cdot, t)$ is the scalar curvature of   $g_{q_i}(t)$. Then by  (\ref{harnack-esti}),  $R_{q_i}(x,t)\leq 2$ whenever $t\in (-s_i,0]$ and $d_{q_i}(q_i, x)\leq   R(q_i)^{\frac{1}{2}}s_i$, where $d_{q_i}(q_i,\cdot)$ is the distance function from  $q_i$ w.r.t   $g_{q_i}(t)$.  It follows that
    $(M, g_{q_i}(t);  q_i)$ with $t\in(-s_i,0]$ converges subsequently to a limit ancient  $\kappa$-solution $(M_\infty, g_\infty(t); q_\infty)$. Moreover,  by (\ref{growth-infty}) and the curvature condition  $\rm{Km}\geq 0$ on $M\setminus K$, one can construct a geodesic line on $(M_\infty, g_\infty(0); q_\infty)$ (cf. \cite[Theorem 5.35]{MT}).  Thus, by the Cheeger-Gromoll  splitting theorem,   $(M_\infty, g_\infty(t); q_\infty)$ is in fact  a splitting  ancient flow  $(N' \times \mathbb{R}, h'(t)+ ds^2;  q_\infty)$, where $(N', h'(t); q_\infty)$ is an  $(n-1)$-dimensional $\kappa$-noncollapsed ancient solution.   Clearly,   $(N', h'(0); q_\infty)$ can not be flat  since $R_\infty(q_\infty,0)=1$, so  $(M_\infty, g_\infty(t); q_\infty)$  is a non-flat ancient solution.  Hence, by  \cite[Proposition 41.13]{KL},  the asymptotic volume ratio of   $(M_\infty, g_\infty(t); q_\infty)$  must be zero.   This  implies    (\ref{AVR}) by the  monotonicity of   volume since the   (\ref{AVR}) is invariant under the rescaling.

  Next we let
  \begin{align}\label{maximal-r}r(p_i)=\sup\{\rho|~{\rm Vol}(B(p_i, \rho))\ge \frac{\omega}{2}\rho^n\}.
  \end{align}
  We prove
  \begin{align}\label{r-R-comparison}
  C_0^{-1}r(p_i)\le R^{-\frac{1}{2}}(p_i)\le C_0r(p_i).
  \end{align}
  % In particular, (\ref{r-R-comparison}) holds at $x=p_i$.
   %$B(p_i, \frac{D}{2} r(p_i))$.
 In fact,  for the  first inequality in  (\ref{r-R-comparison}),  by the volume comparison,  there is $C_1(D)>0$ for any $D>0$ such that
    $$\operatorname{Vol}(B(x, r(p_i))) \geq C_1^{-1} r(p_i)^n,  ~ \forall ~x \in B(p_i, D r(p_i)).$$
   Then by \cite[Lemma 3.5]{DZ-TAMS},   there is  $C_0(D)>0$ such that
  \begin{align}\label{R-bound-volume}
      R \leq C_0^2 r(p_i)^{-2}, \forall~x\in B(p_i, \frac{D}{2} r(p_i)).
  \end{align}
   Thus we only need to  prove the second inequality of (\ref{r-R-comparison}).

    We  use the above argument in the proof of (\ref{AVR}).   On the contrary,  there is a sequence $p_i \rightarrow \infty$ (still denoted by $\{p_i\}$)  such that
  \begin{align}\label{curvature-less-volume}
      \lim _{i \rightarrow \infty} \frac{R^{-1 / 2}(p_i)}{r(p_i)}=\infty.
  \end{align}
  On the other hand, by $(\ref{R-bound-volume})$ and $(\ref{harnack-esti})$, we have
   $$R(x, t) \leq C_0 r(p_i)^{-2}, ~\forall ~x \in B(p_i, \frac{D}{2} r(p_i)), t \in(-\frac{D}{2}, 0].$$
   Then the rescaled flow $\left(M, r(p_i)^{-2} g(r(p_i)^{2}t); p_i\right)$ converges subsequently to a limit ancient solution $(M_\infty', g_\infty'(t); p_\infty')$.
   Note that  $r(p_i)<\infty$ for each $p_i$   by $(\ref{AVR})$. Moreover,
    by the volume comparison, it follows
     \begin{align}\label{r-d-order}\lim _{i \rightarrow \infty} \frac{r(p_i)}{d(p_i,o)}=0.
     \end{align}
  Hence,   by (\ref{r-d-order}) and the curvature condition  $\rm{Km}\geq 0$ on $M\setminus K$, one can construct a geodesic line on $(M_\infty', g_\infty'(0); p_\infty')$ (cf. \cite[Theorem 5.35]{MT}), and so   $(M_\infty', g_\infty'(t); p_\infty')$ is   a splitting  ancient flow  $(\hat N \times \mathbb{R}, \hat h(t)+ ds^2;  p_\infty)$,
       where $\hat {h}(t)$ is an $(n-1)$ dimensional ancient $\kappa$-solution.  As a consequence,  by $(\ref{curvature-less-volume})$, we have
       \begin{align}\label{R-p-flat}
           R_\infty(p_\infty', 0)=0.
       \end{align}

        By the strong maximum principle and (\ref{R-p-flat}),    $(N, \hat {h}(0))$ is flat and so is $(M_\infty', g_\infty'(0))$.  Then by the  injective radius  estimate (cf. \cite[Lemma 3.6]{DZ-TAMS}),  one can show that $(M_\infty', g_\infty'(0))$ must be  isometric to the Euclidean space. In particular,  $\rm{Vol}(B_{g_\infty'(0)}(p_\infty',1))=\omega$.   But this  is impossible by (\ref{maximal-r}).  Hence we finish the proof of (\ref{r-R-comparison}).

        At last,
        by $(\ref{r-d-order})$ and (\ref{r-R-comparison}), we have
      \begin{align}\label{r-decay-upper}\lim _{i \rightarrow \infty} \frac{R(p_i)^{-1}}{d^2(p_i,o)}=0.
     \end{align}
       Then  instead  of the rescaled flow $\left(M, r(p_i)^{-2} g(r(p_i)^{2}t); p_i\right)$  by  $(M,  g_{p_i}(t); p_i)$, the limit  ancient solution $(M_\infty, \bar g(t); p_\infty)$ will split off a line as $(M_\infty', g_\infty'(0); p_\infty')$.
        Thus  $\bar g(t)$  is of  the form  (\ref{r-d-order}).

      In the case of $n=4$, we note that both split $3d$ $\kappa$-noncollapsed ancient flows $h'(t)$ and $\hat h(t)$  in the above arguments are  non-negatively curved when  ${\rm K_m}\ge 0$ away from $K$. Thus both of $h'(t)$ and $\hat h(t)$ are  same as   ancient  $\kappa$-solutions.  Hence  the proofs  above  work  for  the $4d$ steady Ricci solitons when the assumption ${\rm R_m}\ge 0$ is replaced by ${\rm K_m}\ge 0$  away from $K$.

 \end{proof}

According to  the proof in Proposition \ref{Dimension-reduction}, we also get the following curvature comparison.

 \begin{lem}\label{compact-curvature-bound}
     Let $(M^n,g)$ be a complete  noncompact $\kappa$-noncollapsed steady  Ricci soliton  as in Proposition \ref{Dimension-reduction}.  Let $\{p_i\} \to\infty$  be any sequence in   $(M^n,g)$.
         Then for any $q_i\in B_{g_{p_i}}(p_i, D)$,  there is  $C_0(D)>0$ such that
     \begin{align}\label{curvature-control}
         C_0^{-1}R(p_i)\leq R(q_i) \leq C_0 R(p_i).
     \end{align}
 \end{lem}

 \begin{proof}We note that the  rescaled   flow  $(M, g_{p_i}(t); p_i)$  converges  to a  splitting   of ancient solution
  $(M_\infty,  \bar g(t)= h(t)+ds^2;  p_\infty)$. Then   by (\ref{r-R-comparison}) and (\ref{R-bound-volume}),  we get
      the second inequality  of  (\ref{curvature-control}) immediately.  Thus  we  only need to prove the first inequality.

     By contradiction,   there exists a sequence of points $q_i\in B_{g_{p_i}}(p_i, D)$ for some $D>0$ such that
     \begin{align}\label{curvature-flat}
         \frac{R(q_i)}{R(p_i)}\rightarrow 0,~{\rm as}~i\to\infty.
     \end{align}
     Then
     $$R_{\bar g(0)}(q_\infty)=0, $$
     where $q_\infty$ is a limit of $\{q_i\}$ from  the convergence of $(M, g_{p_i}(t); p_i)$.
       By the strong maximum principle,  it follows that $\bar g(0)$ is a flat metric,  which contradicts to the fact  $R_{\bar g(0)}(p_\infty)=1$.
    Thus  (\ref{curvature-control}) holds.
 \end{proof}

 \section{ Compact case of $(N, h(t))$}

 In this section,  we assume that
  %$(M^n, g)$ is  a noncompact $\kappa$-noncollapsed steady   Ricci soliton with  $\rm{Rm}\geq 0$  away from a compact set $K$ of $M$, and
 %consider the case for  Theorem \ref{main-theorem}  that
  there exists a     sequence   of   $p_i\rightarrow \infty$ on   an $n$-dimensional   steady  Ricci soliton   $(M^n, g)$  such that     the corresponding  split   ancient $\kappa$-solution  $(N, h(t))$  of  $(n-1)$-dimension  in Proposition \ref{Dimension-reduction}  satisfies
  \begin{align}\label{bound-h} {\rm{Diam}}(N, h(0))\leq C.
    \end{align}
  % Our purpose is  to show that   (\ref{bound-h}) still holds   all   other split  ancient $\kappa$-solutions  $h'(t)$  of  $(n-1)$-dimension.
 We will  use the method in \cite{DZ-JEMS,DZ-SCM} to study level set  geometry of $(M^n, g)$  under the condition  (\ref{bound-h}).   All  results    in this section holds for any dimension.
 %other split  ancient $\kappa$-solutions  $h'(t)$  of  $(n-1)$-dimension

Firstly  we  show that    $(M,g)$ has  a  convexity property  in sense of the  geodesics.

 \begin{lem}\label{geodesic-away-cpt-set}
     Suppose that  there exists a     sequence   of   $p_i\rightarrow \infty$ such that    the split $(n-1)$-dimensional  ancient  $\kappa$-solution $(N,  {h}(t))$  in  Proposition \ref{Dimension-reduction} satisfies
(\ref{bound-h}).
    Then there exists a compact set $K'$ $(K\subset K'$)  such that for $x_1, x_2\in M\setminus K'$  the minimal  geodesic curve  $\sigma(s)$ connecting $x_1$ and $ x_2$, is contained in    $\subset M\setminus K$,
     where $K$ is the compact set in Proposition  \ref{Dimension-reduction}.
 \end{lem}

 \begin{proof} By the convergence  of  $(M,g_{p_i}(t); p_i)$ together with   (\ref{bound-h}),  it is easy to see that one can choose a point $p\in  \{p_i\}$ such that $B_g(p_i,10C R(p_i)^{-\frac{1}{2}})$ divides $M$ into three parts  with a compact part $\Sigma_{p}$  containing  $K$ as follows,
 \begin{align}\label{m-decomp}M=B_g(p_i,10C R(p_i)^{-\frac{1}{2}}) \cup \Sigma_{p}\cup M',
 \end{align}
 where $B_g(p_i,10C R(p_i)^{-\frac{1}{2}})\cap K=\emptyset$ and  $M'=M\setminus (B_g(p_i,10C R(p_i)^{-\frac{1}{2}}\cup \Sigma_{p})$  is a noncompact set of $M$.
 Set
 $$K'=\Sigma_{p}\cup B_g(p,10C R(p)^{-\frac{1}{2}}).$$
 We need to  verify  that $K'$ is  as required in the lemma.

 On the contrary,  there will exist  two points $x_1, x_2\in M\setminus K'$  and   another point  $x\in \sigma(s)\cap K$, where  $\sigma(s)$ is the
minimal  geodesic  curve connecting  $x_1$ and  $x_2$.  Then   $\sigma(s)$ will pass through $B_g(p,10C R(p)^{-\frac{1}{2}})$ at least twice.  Denote  the first point by $q_1$   and  the last point by $q_2$  in $B_g(p,10C R(p)^{-\frac{1}{2}})$ respectively,    which intersects   with $\sigma(s)$.   Let  $\sigma' $ be the part of  $\sigma(s)$ between $q_1$ and $q_2$. Thus by the triangle inequality,  we have
     \begin{align}\label{sigma-length}
         d_g(q_1,q_2)&=\rm{Length}(\sigma')=d_g(q_1,x)+d_g(x,q_2)\notag\\
         &\geq d_g(q_1,o)+d_g(o,q_2)-2d_g(o,x)\notag\\
         &\geq 2d_g(o,p)-d_g(q_1,p)-d_g(q_2,p)-2C'\notag\\
         &\geq 2d_g(o,p)-20 R(p)^{-\frac{1}{2}}-2C'.
        \end{align}

 On the other hand,
     by the estimate  $(\ref{r-decay-upper})$,   we see that for any small $\delta$ it holds
     \begin{align}\label{p-ASCR}
         R(p_i)^{-\frac{1}{2}}\leq \delta d_g(p_i,o),
     \end{align}
     as long as $i>>1.$
     By  (\ref{sigma-length}), it  follows
  $$d_g(q_1,q_2)\ge 2d_g(o,p)-20 \delta d_g(o,p)-2C'\geq d_g(o,p).
   $$
    However,
   \begin{align}         d_g(q_1,q_2)&\leq d_g(q_1,p)+d_g(p,q_2)\leq 20 R(p)^{-\frac{1}{2}}\notag\\
         &\leq 20 \delta d_g(o,p)\leq \frac{1}{2} d_g(o,p).\notag
     \end{align}
     Thus we get  a contradiction! The lemma is proved.
 \end{proof}

\subsection{Curvature decay estimate}

By Lemma \ref{compact-curvature-bound}
 and  Lemma \ref{geodesic-away-cpt-set}, we prove

 \begin{lem}\label{compact-R-decay}
     Let $(M^n, g)$ be  the steady  Ricci soliton   in  Proposition \ref{Dimension-reduction} with $\rm{Ric}>0$ away from $K$.
       Suppose that  there exists a     sequence   of   $p_i\rightarrow \infty$ such that    the split $(n-1)$-dimensional  ancient  $\kappa$-solution $(N,  {h}(t))$  in  Proposition \ref{Dimension-reduction} satisfies
(\ref{bound-h}).     Then the  curvature of $(M^n,g)$   decays to zero uniformly.   Namely,
    \begin{align}\label{R-decay} \lim_{x\rightarrow \infty}R(x)=0.
    \end{align}
 \end{lem}

 \begin{proof}
     First we  prove
     \begin{align}\label{pi-decay}
         \lim_{p_i\rightarrow \infty}R(p_i)=0.
     \end{align}
     On the contrary,    we assume that $R(p_i)\geq c$ for some constant $c>0$.
      We consider a sequence  of functions    $f_{p_i}=f-f(p_i)$ on the Riemannian  manifolds  $(M, g_{p_i}(0); p_i)$.   By $(\ref{scalar-equ})$, it is easy to see
     \begin{align*}
         |\nabla f_{p_i}|_{g_{p_i}}\leq  c^{-\frac{1}{2}}.
     \end{align*}
     Thus for any $D>0$ it holds
          \begin{align*}
         |f_{p_i}(x)|\leq 2c^{-\frac{1}{2}}D, ~\forall x\in B_{g_{p_i}(p_i,D)}.
  \end{align*}
    By  the regularity of the  Laplace equation,
    $$\Delta_{g_{p_i}} f_{p_i} = R(g_{p_i(0)}),$$
   $f_i$ converges subsequently   to a smooth  function $f_{\infty}$ on $N\times \mathbb{R}$ which
    satisfies the gradient  steady  Ricci soliton equation,
     $$ {\rm{Ric}}(\bar g(0)) =\nabla^2 f_{\infty}.$$
    Note that  $\bar g(0)= h(0)+ds^2$ is a   product  metric. Hence    $(N, h(0))$ is  also a steady gradient Ricci soliton,

On the other hand,  by the maximum principle,  $(N, h(0); p_\infty)$  should be Ricci-flat.   However,  by the normalization of
      $$R(g_{p_i}(0))  ( p_i)=1,$$
     $R( h(0))( p_\infty)$ is also $1$. This is  a contradiction!  (\ref{pi-decay}) is proved.

   By (\ref{pi-decay}) and Lemma \ref{compact-curvature-bound},  we get
     \begin{align}\label{geodesic-ball-decay}
         \lim_{i\to\infty}\sup_{B_{g}(p_{i},10CR(p_{i})^{-\frac{1}{2}}) } R(x)=0.
     \end{align}
   Next we use  (\ref{geodesic-ball-decay}) to derive (\ref{R-decay}).

    Recall  that the set of   equilibrium points of $(M, g, f)$ is  given  by
    \begin{align*}
       S:= \{x\big| |~\nabla f|(x)=0\}.
    \end{align*}
  In general,   $S$ may be non-empty.  But we have

  $\mathbf {Claim 1}$:
       There is no   equilibrium points  away from a compact set  $\hat K$ of $M$ which  contains   $K'$.  Here $K'$ is the set of $M$ determined in   Lemma  \ref{geodesic-away-cpt-set}.

       If $S$ is not empty and $\mathbf {Claim 1}$ is not true,   there will be    two  equilibrium points  $x_1$ and $x_2$ and   a compact set  $\hat K$   containing  $K'$  such that  $x_1, x_2\in M\setminus \hat K$.  Then by  Lemma  \ref{geodesic-away-cpt-set},  there is  a  minimal geodesic  curve $\sigma(s)$  connecting  $x_1$ and $x_2$ such that $\sigma(0)=x_1 $ and $\sigma(T)=x_2 $ and  $\sigma(s)\subset  M\setminus K$.
   Note
     \begin{align*}
         \frac{d}{d s}(\langle \nabla f, \sigma')\rangle(\sigma(s)))=\nabla^2 f(\sigma',\sigma')(s)=\rm{Ric}(\sigma',\sigma').
     \end{align*}
 Thus   we get
     \begin{align*}
         0&= \langle \nabla f, \sigma')\rangle(\sigma(t))-\langle \nabla f, \sigma')\rangle(\sigma(0))\\
         &=\int_0^T \rm{Ric}(\sigma',\sigma')  ds>0,
     \end{align*}
     which is a contradiction! Hence,  $\mathbf {Claim 1}$ is true.

By (\ref{r-decay-upper}), we can choose a subsequence of $\{p_i\}$, still denoted by  $\{p_i\}$ such that
\begin{align}\label{non-intersection}B_{g}(p_i,10CR(p_i)^{-\frac{1}{2}})\cap B_{g}(p_{j},10CR(p_{i+1})^{-\frac{1}{2}})=\emptyset,~\forall ~i,j>>1.
\end{align}
Then as in (\ref{m-decomp}), there are a compact set $\bar K$  and  a sequence of compact set $\{K_i\}$  $(i\ge i_0)$ of $M$ such that $\hat K\subset \bar K$ and
$$\partial K_i\subset \partial B_g(p_i,10C R(p_i)^{-\frac{1}{2}}) \cup \partial B_g(p_{i+1},10C R(p_{i+1})^{-\frac{1}{2}}), $$
and  $M$ is decomposed as
 \begin{align}\label{m-decomp-2}M=\bar K \cup_{i\ge i_0} (K_i
 \cup (B_g(p_{i+1},10C R(p_{i+1})^{-\frac{1}{2}})),
 \end{align}

  $\mathbf {Claim 2}$:
     For any $q_i\in K_i$, there exists  $t_i>0$ such that
     \begin{align}\label{qi-} q_i^{t_i}=\phi_{t_i}(q_i)\in B_{g}(p_i,10R(p_i)^{-\frac{1}{2}}C)\cup B_{g}(p_{i+1},10R(p_{i+1})^{-\frac{1}{2}}C).
     \end{align}

     On the contrary, we see that $\phi_t(q_i)\subset K_i$ for all $t\ge 0$.
       Since $K_i$ is compact, there exists  $c'>0$ by   $\mathbf {Claim 1}$  such that
     $$\rm{Ric}\geq c' g,  c'^{-1}\leq |\nabla f|\leq c'.$$
      It follows
     \begin{align}\label{sequence-increasing}
         \frac{d}{d t} R(\phi_t(q_i))&=-\langle \nabla R, \nabla f\rangle(\phi_t(q_i))= 2\rm{Ric}(\nabla f,\nabla f)\notag\\
         &\geq 2c'^{-1}>0, ~\forall ~t\ge 0.
     \end{align}
     As a consequence,
     $$ R(\phi_t(q_i))\ge 2c'^{-1}t\to\infty,~ {\rm as}~t\to\infty.$$
       This  is impossible since $R(\cdot)$ is uniformly bounded.  Hence, $\mathbf {Claim 2}$ is true.

      By $\mathbf {Claim 2}$ and (\ref{sequence-increasing}),   for any $q_i\in K_i$ we see
   \begin{align}  &R(q_i)\le  R(q_i^{t_i})\notag\\
   &\le \max \{R(x)|~ x\in  B_{g}(p_i,10R(p_i)^{-\frac{1}{2}}C)\cup B_{g}(p_{i+1},10R(p_{i+1})^{-\frac{1}{2}}C)\}.\notag
    \end{align}
  Thus we get (\ref{R-decay}) from  (\ref{geodesic-ball-decay}) and (\ref{m-decomp-2}) immediately.

 \end{proof}

 \begin{rem}\label{f-linear-growth}
 The curvature of steady  Ricci soliton in Lemma \ref{compact-R-decay} decays  uniformly to zero. Then   $|\nabla f(x)|\to 1$ as $\rho(x )\to\infty$  by  (\ref{scalar-equ}).  Moreover,  by   \cite[Lemma 2.2]{DZ-SCM} (or   \cite[Theorem 2.1]{CDM}),   $f$ satisfies (\ref{linear-f}).
      Hence,  the integral curve $\gamma(s)$   generated by $\nabla f$   extends  to the  infinity as $s\rightarrow \infty$.
 \end{rem}

\subsection{Estimate of level sets}

By  Lemma \ref{compact-R-decay} and (\ref{scalar-equ}), there exists a point $p_0\in M $ such that
$$R_{max}=\sup_{p\in M} R(p)=R(p_0)=1.$$
For any positive $c<1$, we set
 $$S(c)=\{p\in M|~R(p)\geq R_{max}-c\}.$$
 Then $S(c)$ is a  compact set.  Moreover, by Remark \ref{f-linear-growth}  there exists a constant  $c_0$ such that $\hat K\subset S(c_0)$ and $\nabla f\neq 0$ on  $ S(c_0)\setminus \hat K$. Thus  VF $\hat X=\frac{\nabla f}{|\nabla f|}$ is well-defined on $S(c)\setminus \hat K$ for any $c\ge c_0$.

  By \cite[Lemma 2.2, 2.3]{CDM},  it is known that there exists a $t_q$ such that $\phi_{t_q}(q)\in S(c_0)$    for any $q\in M\setminus S(c_0)$. Consequently, for any   integral curve of $\hat X=\frac{\nabla f}{|\nabla f|}$, $\Gamma(s) :[0,\infty)\rightarrow M$,  we can reparametrize $s$ such that $\Gamma(0)=p\in S(c_0)\setminus \hat K$,  and   so  $\Gamma(s)  \subset M\setminus \hat K$  is a smooth curve for any $s>0$.

  \begin{lem}\label{split-line}
       Let $(M^n, g)$ be an $n$-dimensional steady  soliton as in Lemma \ref{compact-R-decay} and let  $\Gamma(s)$ be  any   integral curve of $\hat X$ with $\Gamma(0)=p\in S(c_0)\setminus \hat K$.  Then     for any $\epsilon$, there exists a uniform constant $C=C(\epsilon)>0$
        such that
       \begin{align}\label{potential-distance}
           (1-\epsilon)(s_2-s_1)\leq d(\Gamma(s_2),\Gamma(s_1))\leq (s_2-s_1), ~\forall  s_2>s_1>C.
       \end{align}
       In particular,
 \begin{align}\label{potential-distance-special}        (1-\epsilon)s\leq d(\Gamma(s),p)\leq s, ~\forall ~s>C.
   \end{align}

   \end{lem}

   \begin{proof}
    Firstly  by  Remark \ref{f-linear-growth}, we note that for any $\epsilon>0$ there exists a compact set $S'$ such that
    \begin{align}\label{gradient-f-almost-1}
        |\nabla f|(x)>1-\epsilon,~\forall x\in M\setminus S'.
    \end{align}
Moreover,  $(\ref{gradient-f-almost-1})$ holds whenever  $f(x)>L$.
Since  $\Gamma(s)\subset  M\setminus \hat K$,    $|\nabla f|(\Gamma(s))\geq c_0>0$  by $(\ref{R-monotonicity})$ for all $s\geq 0$.  It follows
       \begin{align*}
         f(\Gamma(s))-f(\Gamma(0))=\int_0^s \frac{d}{dt}f(\Gamma(t)) dt=\int_0^s |\nabla f|(\Gamma(t)) dt\geq cs.
     \end{align*}
     Thus there exists a uniform constant $C=\frac{L}{c}+1$ such that $(\ref{gradient-f-almost-1})$ holds as long as $s>C$.

        Let $\gamma:[0, D] \rightarrow M$ be a minimal  geodesic from $\Gamma(s_1)$ to $\Gamma(s_2)$, where $D=d(\Gamma(s_1), \Gamma(s_2))$.  Then by
      $$\frac{d}{d r}\langle\nabla f, \gamma^{\prime}(r)\rangle=\nabla^2 f(\gamma^{\prime}(r), \gamma^{\prime}(r)) \geq 0,$$
        we obtain
     $$
     f(\Gamma(s_2))-f(\Gamma(s_1))=\int_0^D\langle\nabla f, \gamma^{\prime}(r)\rangle d r \leq D\langle\nabla f, \gamma^{\prime}(D)\rangle.
     $$
     This implies
     \begin{align}\label{potential-less-geodesic}
         f(\Gamma(s_2))-f(\Gamma(s_1)) \leq d(\Gamma(s_1), \Gamma(s_2)).
     \end{align}
On the other hand, by $(\ref{gradient-f-almost-1})$,
     we have
     \begin{align}\label{potential-large-geodesic}
         f(\Gamma(s_2))-f(\Gamma(s_1))&=\int_{s_1}^{s_2}\langle\nabla f, \Gamma^{\prime}(r)\rangle d r\notag\\
         &=\int_{s_1}^{s_2}|\nabla f|(\Gamma(r)) d r \geq(1-\epsilon)(s_2-s_1).
     \end{align}
     Thus the first inequality in $(\ref{potential-distance})$ follows from $(\ref{potential-less-geodesic})$ and $(\ref{potential-large-geodesic})$ immediately.
      Note that
     \begin{align}\label{geodesic-dominate}
         d(\Gamma(s_1),\Gamma(s_2))\leq {\rm{Length}}(\Gamma(s))|_{s_1}^{s_2}= s_2-s_1.
     \end{align}
    Hence,   the second inequality in $(\ref{potential-distance})$ also holds.  $(\ref{potential-distance-special})$ is a direct consequence of $(\ref{potential-distance})$ by the triangle inequality.

   \end{proof}

As in Lemma \ref{split-line},   let $\Gamma_i(s)$ be  an  integral curve of $\hat X$ through $p_i$  with  $\Gamma_i(0)\in S(c_0)$ and $\Gamma_i(s_i)=p_i$. For any $D>0$,  we  set
 $$\hat \Gamma_i(s)=\Gamma_i(R(p_i)^{-\frac{1}{2}} s+s_i), ~s\in [-D, D].$$
 Then it is easy to see
 $$|\frac{d\hat \Gamma_i(s)}{ds}|_{g_{p_i}(0)}= 1, ~s\in [-D, D].$$
 Thus
 $\hat \Gamma_i(s)$ is an   integral curve of $\hat X_i= \frac{\nabla_i f}{|\nabla_i f|}$ through $p_i$,
where    $\nabla_i$ is the gradient operator w.r.t. the metric  $(M,  g_{p_i}(t);p_i)$.

 With the help of Lemma \ref{split-line},  we prove  that the splitting line obtained by Proposition \ref{Dimension-reduction} is actually a limit of a family of  integral curves of $\hat X_i$ under the condition in Lemma \ref{compact-R-decay}.

 \begin{lem}\label{integral-curve-converge}
     Let $(M^n, g)$ be  the  steady soliton in Lemma \ref{compact-R-decay} and  $(N\times \mathbb{R},h(t)+ds^2;p_\infty)$ the splitting limit flow  of  $(M,g_{p_i}(t);p_i)$.  Then $\hat \Gamma_i(s)$    converges   locally  to a  geodesic line on $N\times \mathbb{R} $ w.r.t.  the metric  $(M,  g_{p_i}(t);p_i)$.
 \end{lem}

 \begin{proof}
     Since  $X_i=\nabla_i f$  is convergent  w.r,t. the metrics $(M,  g_{p_i}(t);p_i)$ (cf.  \cite[Lemma 4.6]{DZ-JEMS}),   $\hat{X}_{i}$ also converges subsequently to a VF $\hat{X}_{\infty}$ on $(N\times \mathbb{R},h(t)+ds^2;p_\infty)$.
   Thus  $\hat \Gamma_i(s)$ converges to an  integral curve $\hat \Gamma_{\infty}(s)$ of $\hat{X}_{\infty} $ on  $N\times \mathbb{R}$, where $s\in (-\infty, \infty)$.  It remains to show that  $\hat \Gamma_{\infty}(s)$ is a line.

  %   which is the integral curve of $\hat{X}_{i}=R(p_i)^{-\frac{1}{2}}\hat{X}$. Then argue as in \cite[Lemma 4.6]{DZ-JEMS},
  Since $p_i\to\infty$, we have $s_i\to \infty$. Then  by $(\ref{AVR})$ and $(\ref{r-R-comparison})$ , for any number $D>0$, it holds
         % $$\lim _{s_i \rightarrow \infty} R(p_i)^{\frac{1}{2}} s_i \rightarrow \infty.$$
           $$s_i-\operatorname{D}R(p_i)^{-\frac{1}{2}} \rightarrow \infty.$$
  By applying  (\ref{potential-distance}) to each  $\hat\Gamma_i(s')$,  we get
   $$(1-\epsilon) D R(p_i)^{-\frac{1}{2}} \le d (\hat \Gamma_i(-D), \hat \Gamma_i(0))\le DR(p_i)^{-\frac{1}{2}}$$
    and
      $$(1-\epsilon) D R(p_i)^{-\frac{1}{2}} \le d (\hat \Gamma_i(D), \hat \Gamma_i(0))\le DR(p_i)^{-\frac{1}{2}}.$$
            It follows
 $$2(1-\epsilon) D R(p_i)^{-\frac{1}{2}} \le d (\hat \Gamma_i(-D), \hat \Gamma_i(D))\le 2DR(p_i)^{-\frac{1}{2}},$$
and consequently,
$$2(1-\epsilon) D \le d_{g_{p_i}}  (\hat \Gamma_i(-D), \hat \Gamma_i(D))\le 2D.$$
 Thus  by taking the limit  of $\hat \Gamma_i(s)$  as  well as $\epsilon\to 0$,  we obtain
 \begin{align}\label{distance-D} d_{g_\infty}\left(\hat\Gamma_{\infty}(-D), \hat\Gamma_{\infty}(D)\right)=2 D.
 \end{align}
 Note that $2D$ is  the  length of $\hat\Gamma_{\infty}(s)$ between $\hat \Gamma_{\infty}(-D)$ and $\hat\Gamma_{\infty}(D)$.
 Hence,  $\hat\Gamma_{\infty}(s)$ must be a minimal  geodesic  connecting  $\hat \Gamma_{\infty}(-D)$ and $\hat\Gamma_{\infty}(D)$.
Since    $D$ is arbitrary,        $\hat\Gamma_{\infty}(s)$  can be extended to  a geodesic line.
\end{proof}

 Now we  begin to prove  main results  in this subsection.

 \begin{lem}\label{compact-ancient-solution} Let $(M^n, g)$ be  the  steady soliton in Lemma \ref{compact-R-decay} and  $(N\times \mathbb{R},h(t)+ds^2;p_\infty)$ the splitting limit flow  of  $(M,g_{p_i}(t);p_i)$, which satisfies (\ref{bound-h}).
   Then
     $f^{-1}(f(p_i))\subseteq B_{g_{p_i}}(p_i,200C)$ when $i>>1$.
  \end{lem}

 \begin{proof}
      On the contrary, there will exist a $q_i'\in \partial B_{g_{p_i}}(p_i,100C)\cap f^{-1}(f(p_i))$ and  a minimal  geodesic $\bar \gamma_i\subset f^{-1}(f(p_i))$ connecting $p_i$ and $q_i'$  w.r.t. the induced metric $\bar g_{p_i}$ on $f^{-1}(f(p_i))$ such that
     $$\bar \gamma_i\subset B_{g_{p_i}}(p_i,100C) .$$
    Then
     \begin{align}\label{length-bar-gamma}
         {\rm{Length}}_{\bar g_{p_i}}(\bar \gamma_i)\geq d_{g_{p_i}}(p_i,q_i')=100C.
     \end{align}
   On the other hand, according to  the proofs in \cite[Lemma 4.3-Proposition 4.5]{DZ-JEMS},  the part $\Sigma_i = f^{-1}(f(p_i))\cap B_{g_{p_i}}(p_i,100C) $ of the  level set $f^{-1}(f(p_i))$,  which contains $\bar \gamma_i$, converges  subsequently to an $(n-1)$-dimensional open manifold $(\Sigma_\infty, h' ; p_\infty)$ w.r.t. the induced metric $\bar g_{p_i}$.
   As a consequence, the  minimal  geodesic $\bar \gamma_i$ converges subsequently to a  minimal  geodesic $\bar \gamma$ in $\Sigma_\infty$.
   Thus by  $(\ref{length-bar-gamma})$, we get
     \begin{align}
         {\rm{Length}}_{h'}(\bar \gamma)\geq 100C.
     \end{align}

     Next we show that $(\Sigma_\infty, h')$ is an open set of   $(N, h(0))$.  Then it  follows
     $${\rm {Diam}}(N,h(0))
     \ge {\rm {Diam}}( \Sigma_\infty, h')\ge  100C,$$
     which contradicts to (\ref{bound-h}). The lemma will be proved.

     Let $\hat X_i=\frac{\nabla_i f}{|\nabla_i f|}$.  By Lemma \ref{compact-R-decay}, $(\ref{gradient-f-almost-1})$ and Shi's estimates we can calculate that
    \begin{align*}
        \sup_{B(p_i,2D)_{g_{p_i}}} |\nabla_i \hat X_i|_{g_{p_i}}&= \sup_{B(p_i,2D)_{g_{p_i}}} R(p_i)^{-\frac{1}{2}}(\frac{|\rm{Ric}|}{|\nabla f|}+\frac{|\rm{Ric}(\nabla f,\nabla f)|}{|\nabla f|^3})\notag\\
        &\leq CR(p_i)^{\frac{1}{2}}\to 0,
    \end{align*}
    and
    \begin{align*}
        \sup_{B(p_i,2D)_{g_{p_i}}} |\nabla^m_i \hat X_i|_{g_{p_i}}&\leq C(m)\sup_{B(p_i,2D)_{g_{p_i}}} |\nabla^{m-1}_i\rm{Ric(g_{p_i})}|_{g_{p_i}}\leq C'.
    \end{align*}
    Thus $\hat X_i$ converges subsequently to a parallel vector field $ \hat X_\infty$ on $(N\times \mathbb{R},h(t)+ds^2;p_\infty)$. Moreover,
    \begin{align*}
        \sup_{B(p_i,2D)_{g_{p_i}}} |\hat X_i|_{g_{p_i}} = 1.
    \end{align*}
   Hence,  $\hat X_\infty$ is a non-trivial parallel vector field on $N\times \mathbb{R}$.

    $\hat X_\infty$ is also perpendicular to  $(\Sigma_\infty, h')$.  In fact, for any  $V\in T\Sigma_\infty$ with $|V|_{h'}=1$, by \cite[Proposition 4.5]{DZ-JEMS}, there is a sequence  of $V_i\in T\Sigma_i$ such that $R(p_i)^{-\frac{1}{2}}V_i\to V$.   Thus
     \begin{align*}
         h'(V, \hat X_{\infty})=\lim_{i\to \infty} g_{p_i}(R(p_i)^{-\frac{1}{2}}V_i, \hat X_i)=\lim_{i\to \infty} g(V_i, \frac{\nabla f}{|\nabla f|})=0.
     \end{align*}

     By Lemma \ref{integral-curve-converge},  we  have already known that  $\hat X_\infty$ generates a geodesic line $\hat \Gamma_\infty$   through   $p_\infty$  on $N\times \mathbb{R}$. Note that  $(N,h(0))$ is compact by $(\ref{bound-h})$.    $\hat X_\infty$  must  be tangent to  the splitting  line direction of $N\times \mathbb{R}$,  and consequently,   $(\Sigma_\infty, h'; p_\infty)\subset ( N, h(0);  p_\infty) $.
     %,  it  is a  parallel  vector field and is  perpendicular to  $(\Sigma_\infty, h')$.
      Namely,  $(\Sigma_\infty, h')$ is an open set of   $(N, h(0))$. The proof is complete.

       \end{proof}

By Lemma \ref{compact-ancient-solution},  we prove

 \begin{prop}\label{compact-levelset-compact}Let $(M^n, g)$ be  the  steady Ricci  soliton in Lemma \ref{compact-R-decay} and  $(N\times \mathbb{R},h(t)+ds^2;p_\infty)$ the splitting limit flow  of  $(M,g_{p_i}(t);p_i)$, which satisfies (\ref{bound-h}).
     Then there exists $C_0(C)>0$  such that for any $q_i\in  f^{-1} (f(p_i))$ the  splitting  limit flow $( h'(t)+ds^2, N'\times \mathbb{R};  q_\infty)$ of rescaled flow $(M,g_{q_i}(t); q_i)$ satisfies
     \begin{align}\label{diam-estimate} {\rm Diam}(h'(0))\leq C_0.
     \end{align}
 \end{prop}

 \begin{proof}The convergence part comes from Proposition \ref{Dimension-reduction}. We need to check (\ref{diam-estimate}).
       In fact,   by  Lemma \ref{compact-ancient-solution} and Lemma \ref{compact-curvature-bound},    there are
     $C_1, C_2>0$ such that  for any $D>0$ such that
     \begin{align*}
         B_{g_{q_i}}(q_i,D)\subset B_{g_{p_i}}(q_i,C_1D)\subset B_{g_{p_i}}(p_i,C_1D+C_2),
     \end{align*}
     where  $g_{q_i}=R(q_i)g $.  Similarly,  we have
     \begin{align*}
         B_{g_{q_i}}(q_i,D)\supset B_{g_{p_i}}(q_i,C_1^{-1}D)\supset B_{g_{p_i}}(p_i,C_1^{-1}D-2C_2).
     \end{align*}
     Then it  is easy to see that the splitting Ricci flow $(h'(t)+ds^2, N'\times \mathbb{R}; q_\infty)$ of  $(M,g_{q_i}(t);q_i)$   is isometric to $(h(t)+ds^2, N\times \mathbb{R}; p_\infty)$ up to scaling.   As a consequence, we get
     \begin{align*}
         {\rm{Diam}}(h'(0))\leq (C_1+10){\rm{Diam}}(h(0)) \leq (C_1+10)C.
     \end{align*}
   The proposition  is proved.

 \end{proof}

\begin{rem}\label{N-levelset}  The arguments in the proofs of  Lemma \ref{compact-ancient-solution} and Proposition \ref{compact-levelset-compact} also  imply that both of $N$ and $N'$ are diffeomorphic to each  level set  $f^{-1} (f(p_i))$ when $i>>1$.
In fact, the submanifolds $(f^{-1}(f(p_i)), \bar g_{p_i})$ ($(f^{-1}(f(p_i)), \bar g_{q_i})$)  converge to $(N, h(0))$ ($(N', h'(0))$) in   the Cheeger-Gromov sense.

 \end{rem}

 \section{$4d$ steady  Ricci solitons}

 In this section, we first recall recent works on  compact $3d$ ancient   $\kappa$-solitons by Angenent-Brendle-Daskalopoulos-Sesum and  Brendle-Daskalo-poulos-Sesum \cite{ABDS, BDS}, then we estimate the diameter of $(N, h(0))$  for any  split limit flow $(N,h(t))$ of  the rescaled flow sequence.

  As we know,     Perelman's  model of ancient solution is of  type II, which   is defined on   $S^3$  with $Z_2\times O(2)$-symmetry for any $t\in (-\infty, 0)$  \cite{P2}.   According to  \cite{Ham-singular},  we have the following  definitions.

    \begin{defi}
  An  ancient solution with ${\rm K_m}\ge0$ is called
    type \uppercase\expandafter{\romannumeral1} if it satisfies
         \begin{align*}
             \sup_{M\times (-\infty,0]} (-t)R(x,t)<\infty.
         \end{align*}
       Otherwise,   it is called type \uppercase\expandafter{\romannumeral2},  i.e.,   it satisfies
         \begin{align*}
             \sup_{M\times (-\infty,0]} (-t)R(x,t)=\infty.
         \end{align*}
     \end{defi}

   Fix $p_0\in S^3$.   We normalize the Perelman  solution by
  \begin{align}\label{perelman-solutiuon-normalization}R_{max}(-1)=R(p_0,-1)=1.
  \end{align}
  For simplicity, we denote it by $(S^3,g_{Pel}(t);p_0)$, $t\in (-\infty,0)$.

  The asymptotic behavior of  $(S^3,g_{Pel}(t);p_0)$ has been computed in  \cite{ABDS} as follows,
 \begin{align}\label{asymptotic-behaviour}
     &{\rm{Diam}}({g_{Pel}(t)}) \geq 2.1\sqrt{(-t)\log(-t)},\notag\\
     &R_{max} \leq 1.1\frac{\log(-t)}{-t},\notag\\
     &R_{min}\geq \frac{C}{-t}.
 \end{align}
  Here $-t\geq L$ for some large $L>10000C>10000$.
  In particular,
 there exists a constant $C_{Diam}$ such that ${\rm{Diam}}(g_{Pel}(t))\geq C_{Diam}$, when $-t\geq 2L$.
 Moreover,
   $${\rm{Diam}}(g_{Pel}(t))R^{\frac{1}{2}}(q,t) $$
    is strictly increasing as
   $t\rightarrow -\infty$,
   and
 \begin{align}\label{Perelman-type2}
     \lim_{t\rightarrow -\infty}{\rm{Diam}}(g_{Pel}(t))R^{\frac{1}{2}}(q,t)=\infty,~\forall~ q\in S^3.
 \end{align}
 The above limit is  uniformly.
 Usually, we call   $3d$ ancient   $\kappa$-solitons  of  type II  on $ S^3$ as Perelman  (ancient) solutions.

  The following classification result of  3d  compact  ancient  $\kappa$-solutions of type II was proved in  \cite{BDS}.

 \begin{theo}\label{Brendle} Any 3d compact simply connected  ancient  $\kappa$-solution  of type II coincides with
   a reparametrization in space, a translation in time, and a parabolic rescaling of Perelman   solution  $(S^3,g_{Pel}(t);p_0)$.
 \end{theo}

By Theorem \ref{Brendle},  for  any    simply connected   compact  $3d$  $\kappa$-solution $(M, h(t); q)$ of type II and a point $q\in M$,    there exist  constant  $\lambda$, a time $T$, $p\in S^3$  and a diffeomorphism $\Psi$ from $S^3$ to $M$   such that $\Psi(p)=q$ and
\begin{align}\label{align}
     (\Psi^{-1}(M), \lambda  \Psi^*(h(\lambda^{-1}t)), \Psi^{-1}(q)) =(S^3, g_{Pel}(t-T); p_0).
 \end{align}

We note that  the  Perelman's solution $(S^3,g_{Pel}(t);p_0)$ is $Z_2\times O(2)$-symmetric.  Then the isometric subgroup of $(S^3,g_{Pel}(t);p_0)$   must be  $Z_2\times G$, where $G$ is a subgroup of $O(2)$. Thus   $G$ fixes  the minimal geodesic connecting the  two tips of  the Perelman  solution.  It follows that  any  quotient of Perelman  solution, which is  also an ancient  $\kappa$-solutions,  satisfies  $(\ref{asymptotic-behaviour})$. Hence, by the classification Theorem \ref{Brendle}, we get

 \begin{prop}\label{Perelman-T}
     Let  $(M, h(t))$ be a 3d  compact  ancient    $\kappa$-solutions of type II  and $p\in M$,  which satisfies
     \begin{align}\label{Perelman-normalize-curvature}
         R(p,0)=1
     \end{align}
     and
     \begin{align}\label{Perelman-normalize-bound}
         {\rm{Diam}}(h(0))=\frac{2C'}{3}>10C_{Diam}.
     \end{align}
     Then for any $q\in M$, ${\rm{Diam}}(h(t))R^{\frac{1}{2}}(q,t)$ is strictly decreasing for $t\leq 0$. Moreover,  there exists a  $T(C')$ such that
     \begin{align}\label{diam-radio}{\rm{Diam}}(h(T(C')))R^{\frac{1}{2}}(q,T(C'))=2C'.
     \end{align}
 \end{prop}

 By   Theorem  \ref{Brendle} and Proposition \ref{Perelman-T},  we are able to classify the split ancient $\kappa$-solutions of dimension 3 when
 the  $4d$  noncompact $\kappa$-noncollapsed steady Ricci soliton in Theorem \ref{main-theorem} admits a  split  noncompact ancient $\kappa$-solution $(N, h(t))$.

 We need  the following definition  introduced by Perelman (cf. \cite{P2}).

  \begin{defi}\label{epsilonclose}
      For any $\epsilon>0$, we say a pointed Ricci flow $\left(M_1, g_1(t); p_1\right), t \in$ $[-T, 0]$, is $\epsilon$-close to another pointed Ricci flow $\left(M_2, g_2(t); p_2\right), t \in[-T, 0]$, if there is a diffeomorphism onto its image $\bar{\phi}: B_{g_2(0)}\left(p_2, \epsilon^{-1}\right) \rightarrow M_1$, such that $\bar{\phi}\left(p_2\right)=p_1$ and $\left\|\bar{\phi}^* g_1(t)-g_2(t)\right\|_{C^{\left[\epsilon^{-1}\right]}}<\epsilon$ for all $t \in\left[-\min \left\{T, \epsilon^{-1}\right\}, 0\right]$, where the norms and derivatives are taken with respect to $g_2(0)$.
  \end{defi}

  By   the  above definition and  Proposition \ref{Dimension-reduction},   we  get  immediately,

  \begin{prop}\label{epsilon-close}
      Let $(M^n,g)$ be the  steady  Ricci soliton in Proposition \ref{Dimension-reduction}. Then for any $\epsilon>0$, there exists a compact set  $D(\epsilon)>0$, such that for any $p\in M\setminus D$, $(M,g_p(t);p)$ is $\epsilon$-close to a splitting  flow $(h_p(t)+ds^2; p)$,  where  $h_p(t)$  is an $(n-1)$-dimensional ancient $\kappa$-solution.
  \end{prop}

We note that  for a given  $p$ and a  number   $\epsilon>0$ the  $\epsilon$-close  splitting  flow $(h_p(t)+ds^2; p)$ may not be unique  in   Proposition \ref{epsilon-close}.
   Due to \cite{Yi-flyingwings}, we introduce a function on $M$ for each  $\epsilon$  by
 \begin{align}\label{notation-f}F_{\epsilon}(p)=\inf_{h_p} \{ {\rm{Diam}} (h_p(0))\in(0,\infty)  \}.
 \end{align}
 For simplicity, we always  omit the subscribe $\epsilon$ in the function  $F_{\epsilon}(p)$ below.

By estimating  (\ref{notation-f}), we  prove

 \begin{prop}\label{noncompact-ancient-solution}
     Let $(M^4,g)$ be a noncompact $\kappa$-noncollapsed steady Ricci soliton in Theorem \ref{main-theorem}. Suppose that  there exists a sequence of pointed rescaled Ricci flows  $(M,g_{p_i}(t); p_i)$ , which converge  subsequently to a splitting Ricci flow $(h(t)+ds^2; p_\infty)$ for some noncompact ancient $\kappa$-solution $h(t)$. Then for any limit flow $(h'(t)+ds^2; q_\infty)$ of rescaled Ricci flows  $(M,g_{q_i}(t); q_i)$, $h'(t)$ is a noncompact ancient $\kappa$-solution.
 \end{prop}

 \begin{proof}
     We argue by contradiction. Suppose that there exists a limit flow $(h'(t)+ds^2; q_\infty)$ converged by rescaled flows  $(M,g_{q_i}(t);q_i)$, which satisfies $(\ref{bound-h})$.  Then by Proposition  \ref{compact-levelset-compact}, there exists a uniform constant $C_3(C)>0$, such that
     \begin{align}\label{compact-diameter}
         F(p_i')\leq C_3
     \end{align}
     for all $i$, and all $p_i'\in f^{-1}(f(q_i))$.

     Fix $C'=\max\{100C_{Diam},10C_3\}$ and $T(C')$ as in  Proposition \ref{Perelman-T}.   We choose an $\epsilon>0$ such that $\epsilon^{-1}>\max\{-10T(C'),100C'\}$.
     Thus for the sequence of $(M,g_{p_i}(t); p_i)$ in   Proposition \ref{noncompact-ancient-solution},  we can choose a point
     $p_{i_0}\in\{p_i\}$ such that
     \begin{align}\label{F_0}
         F(p_{i_0})>C'\geq 100C_{Diam}.
     \end{align}
        Let $\Gamma_1$ be the integral curve of $\hat X$ passing through $p_{i_0}$ with  $\Gamma_1(0)=p_{i_0}$, which tends to the infinity by Lemma \ref{split-line}. We claim:
       \begin{align}\label{noncompact-diameter}
         F(\Gamma_1(s))>\frac{1}{2}C', ~ \forall ~s\geq0.
     \end{align}

   Define
     $$s_0=\sup\{s\geq0|F(s')\geq C'~\text{for all}~ s'\in[0,s]\}.$$
 If $s_0=\infty$,   $F(s)>C'$ for all $s\geq 0$.  Then  (\ref{noncompact-diameter}) is obvious true  in this case.
  Thus  we   may consider the case $s_0<\infty$, i.e.,  $F(s_0)=C'$ for some $s'=s_0$.  It follows that  there exists a $3d$ compact
    ancient $\kappa$-solution $h_{\Gamma_1(s_0)}(t)$ such that
   \begin{align}\label{s0-close}
       &  (M,R(\Gamma_1(s_0))g(R(\Gamma_1(s_0))^{-1}t);\Gamma_1(s_0))\notag\\
       & \overset{\epsilon-\text{close}}{\sim} (N\times \mathbb{R},h_{\Gamma_1(s_0)}(t)+ds^2; \Gamma_1(s_0)).
     \end{align}
   Since   the diameter of $h_{\Gamma_1(s_0)}(0)$ is large,  $h_{\Gamma_1(s_0)}(t)$ can not be a family of shrinking quotient  spheres.  Hence, by Theorem \ref{Brendle},  it must be a  quotient of  Perelman  solution after a reparametrization.

     By  Proposition  \ref{Perelman-T},  we  see that  ${\rm{Diam}}(h_{\Gamma_1(s_0)}(t))R_h^{\frac{1}{2}}(\Gamma_1(s_0),t) $ is strictly decreasing for $t\in (-\epsilon^{-1},0]$.  Then by the choice of the number  $s_0$, it follows
     \begin{align}\label{t0-t1-diameter}
        { \rm{Diam}}(h_{\Gamma_1(s_0)}(t))R_h^{\frac{1}{2}}(\Gamma_1(s_0),t)\ge \frac{2C'}{3}, ~t\in (-\epsilon^{-1},0].
     \end{align}
      Moreover,   by the choice of  $T(C')$,   we have
     \begin{align*}
         {\rm{Diam}}(h(T(C'))R_h^{\frac{1}{2}}(\Gamma_1(s_0), T(C'))=2C'.
     \end{align*}
     Let  $t_1=\min\{-1000, T(C')\}\ge -\frac{\epsilon^{-1}}{2}$.  Thus
             \begin{align}\label{t1-diameter}
         {\rm{Diam}}(h_{\Gamma_1(s_0)}(t_1))R_h^{\frac{1}{2}}(\Gamma_1(s_0), t_1)\ge 2C'.
     \end{align}

    Recall that  $\left\{\phi_t\right\}_{t \in(-\infty, \infty)}$ is the flow of $-\nabla f$ with $\phi_0$ the identity and $(g(t), \Gamma_1(s))$ is isometric to $\left(g, \phi_t(\Gamma_1(s))\right)$.  Then
     \begin{align*}
         \phi_t(\Gamma_1(s))=\Gamma_1\left(s-\int_0^t|\nabla f|\left(\phi_\mu(\Gamma_1(s))\right) d \mu\right)
     \end{align*}
     Let $T=t R\left(\Gamma_1\left(s_0\right)\right)^{-1}<0$ and
     \begin{align}\label{s0-s1-distance}
         s=s_0-\int_0^{T}|\nabla f|\left(\phi_\mu\left(\Gamma_1\left(s_0\right)\right)\right) d \mu.
     \end{align}
     Set
    $$ s_1=s_0  -\int_0^{T_1}|\nabla f|\left(\phi_\mu\left(\Gamma_1\left(s_0\right)\right)\right) d \mu,$$
    where
     $T_1=t_1 R\left(\Gamma_1\left(s_0\right)\right)^{-1}.$
     Since  the scalar curvature $R$  of $(M, g)$ decays to $0$ uniformly  by Proposition \ref{compact-R-decay},    we may assume $|\nabla f|\geq \frac{1}{2}$ along $\Gamma_1$.  Thus
     \begin{align}\label{s-1} s_1-s_0
     \geq 500R^{-1 }\left(\Gamma_1\left(s_0\right)\right) \geq 500R^{-1}_{max}=500.
     \end{align}

     Note   that
        $\phi_{T}\left(\Gamma_1\left(s_0\right)\right)=\Gamma_1 \left(s\right)$ and $\left(g\left(T\right), \Gamma_1\left(s_0\right)\right)$ is isometric to $\left(g, \Gamma_1\left(s\right)\right)$ for all $s\in[s_0,s_1]$.  Then
     \begin{align}\label{levelset-isometric}
         (M,R(\Gamma_1(s))g;\Gamma_1(s))&\cong (M,R(\Gamma_1(s_0),T)g(T);\Gamma_1(s_0))\notag\\
         &\cong (M,\frac{R(\Gamma_1(s_0),T)}{R(\Gamma_1(s_0))}R(\Gamma_1(s_0))g(T);\Gamma_1(s_0)).
     \end{align}
      Since
  $R(\Gamma_1(s_0),T)\leq R(\Gamma_1(s_0))$  by $(\ref{R-monotonicity})$,
  we get from (\ref{s0-close}),
     \begin{align}\label{s1-close1}
         &(M,R(\Gamma_1(s))g;\Gamma_1(s))\notag\\
         & \overset{\epsilon-\text{close}}{\sim} (N\times \mathbb{R},\frac{R(\Gamma_1(s_0),T)}{R(\Gamma_1(s_0))}h_{\Gamma_1(s_0)}(t)+ds^2; \Gamma_1(s_0)).
     \end{align}
    On the other hand,  there is another  $3d$ compact
    ancient $\kappa$-solution $h_{\Gamma_1(s_1)}(t)$ corresponding to   the point $\Gamma_1(s)$  such that
     \begin{align}\label{s1-close2}
         (M,R(\Gamma_1(s))g;\Gamma_1(s)) \overset{\epsilon-\text{close}}{\sim} (h_{\Gamma_1(s)}(0)+ds^2, \Gamma_1(s)).
     \end{align}
     Hence, combining $(\ref{s1-close1} )$ and $(\ref{s1-close2} )$, we derive
     \begin{align}\label{two-ancients}
         h_{\Gamma_1(s)}(0) \overset{\epsilon-\text{close}}{\sim} \frac{R(\Gamma_1(s_0),T)}{R(\Gamma_1(s_0))}h_{\Gamma_1(s_0)}(t).
     \end{align}

     By the convergence of $(M,g_p(t);p)$, we have
     $$\frac{R(\Gamma_1(s_0),T)}{R(\Gamma_1(s_0))} \overset{\epsilon-\text{close}}{\sim} R_h(\Gamma_1(s_0),t), ~\forall ~t\in [t_1, 0],$$
     and so,
     \begin{align*}
         {\rm{Diam}}(\frac{R(\Gamma_1(s_0),T)}{R(\Gamma_1(s_0))}h_{\Gamma_1(s_0)}(t))\overset{\epsilon-\text{close}}{\sim} {\rm{Diam}}(R_h(\Gamma_1(s_0),t)h_{\Gamma_1(s_0)}(t))).
      \end{align*}
     Then by (\ref{two-ancients}), the monotonicity  $(\ref{t0-t1-diameter})$ implies that
   \begin{align}\label{s-large}  F(\Gamma_1(s) )\geq \frac{2C'}{3}-2\epsilon>\frac{1}{2}C', ~ \forall ~s \in\left[s_0, s_1\right].
   \end{align}
     Moreover, by $(\ref{t1-diameter} )$,
     \begin{align}\label{F-estimate}   F\left(\Gamma_1(s_1)\right) > 2C'-2\epsilon>C'.
     \end{align}

    By (\ref{F-estimate})  together with (\ref{s-large})  and (\ref{s-1}),  we can   repeat the above argument  to obtain (\ref{noncompact-diameter}). On the other hand,    the curve $\Gamma_1(s)$ passes through level sets $f^{-1}(f(q_i))$  because of   $\lim_{s\rightarrow \infty} f( \Gamma_1(s) )=\infty$.  Thus for each $q_i$ ($i>>1$) there exists $p_i'\in f^{-1}(f(q_i))$ such that  $p_i'=\Gamma_1(s_i)$  for some $s_i$.   By $(\ref{compact-diameter})$,  $F(p_i')\leq C_3$, which contradicts to $(\ref{noncompact-diameter})$.  Hence,   the proposition is proved.
 \end{proof}

 \section{Proofs of the main results}

 In this section, we  prove Theorem \ref{main-theorem} and Corollary \ref{coro}.  Firstly,  we  consider a special case:   there is a uniform constant $C$  such that  all  split ancient $\kappa$-solution $h(t)$   in Proposition \ref{Dimension-reduction}  satisfies   $(\ref{bound-h})$.  By generalizing the argument in Section 3 we prove

 \begin{prop}\label{compact-limit-ancient-solution}
    Let $(M^n, g)$ $(n\ge 4)$  be an  $n$-dimensional    noncompact $\kappa$-noncollapsed   gradient steady Ricci soliton  with  $\rm{Rm}\geq 0$   and $\rm{Ric}> 0$ away from a compact set $K$ of $M$.    Suppose that  there is a uniform constant $C$  such that   all  split   ancient $\kappa$-solution $(N^{n-1}, h(t))$ of $(n-1)$-dimension  in Proposition \ref{Dimension-reduction}  satisfies $(\ref{bound-h})$.
         Then  every $h(t)$ must be  ancient $\kappa$-solutions of type \uppercase\expandafter{\romannumeral1}.  As a consequence,
          \begin{align}\label{N-split}
        &( N^{n-1}, h(t))\notag\\
     &   \cong (N_1,h_1(t))\times\cdots\times (N_k,h_k(t))\times (\hat N_1, \hat h_1(t))\times \cdots \times (\hat N_\ell, \hat h_\ell(t)),
    \end{align}
    where each  $(N_i,h_i(t))$ is a family of shrinking   quotients of a  closed symmetric space with nonnegative curvature operator, and each  $(\hat N_j, \hat h_j(t))$ is  a family of shrinking  round  quotient spheres.
\end{prop}

To prove Proposition \ref{compact-limit-ancient-solution},  we shall exclude the existence of   compact $\kappa$-noncollapsed ancient solutions  of type \uppercase\expandafter{\romannumeral2}.  Our proof   is based on   two lemmas below,
 which can be regarded as  higher   dimensional   versions  of \cite[Lemma 2.1, Lemma 2.2]{ABDS}, respectively.

\begin{lem}\label{noncompact-asymptotic-soliton}
    Let $(N^{n-1}, h(t))$ be  an $(n-1)$-dimensional   compact $\kappa$-solution of   type  \uppercase\expandafter{\romannumeral2}.
    % to the Ricci flow with positive curvature operator.
    Fix $p\in N$,  we consider $t_k\rightarrow -\infty$ and a sequence of points $x_k\in N$ such that $\ell(x_k,t_k)<\frac{n-1}{2}$, where $\ell$ denotes the reduced distance from $(p,0)$.  %Suppose that we .
    Then the rescaled manifold  by dilating  the manifold $(N, h(t_k))$ around the point $x_k$ by the factor $\frac{1}{\sqrt{-t_k}}$  converges to a noncompact shrinking gradient Ricci soliton.
\end{lem}

\begin{proof}
    According to a classification  result of closed ancient $\kappa$-solutions  with nonnegative curvature operator (cf. \cite[Theorem 7.34]{CLN},  \cite{BW}), we have
    \begin{align}\label{M-split}
       & ( N,  h(t))\notag\\
       &\cong (N_1,h_1(t))\times\cdots\times (N_k,h_k(t))\times (\hat N_1, \hat h_1(t))\times \cdots \times (\hat N_\ell, \hat h_\ell(t)),
    \end{align}
 where each  $(N_i,h_i(t))$ is a family of   shrinking   quotients of an $(n-1)$-dimensional   closed symmetric space  with nonnegative curvature operator,  and   each  $(\hat N_j, \hat h_j(t))$  is an ancient flow with positive curvature operator on a quotient of $(n-1)$-dimensional   sphere.
    By Perelman's argument in \cite[Section 11]{P}, the rescaled   manifolds  converge in the Cheeger-Gromov sense to a  $\kappa$-noncollapsed  shrinking  gradient Ricci soliton   $(N', h'(t))$ with nonnegative curvature operator.
    If the limit soliton   $(N', h'(t))$  is compact,      the limit   $(N_j', h_j'(t))$  of
    $$(\hat N_j,(-t_k)^{-1}\hat h_j((-t_k) t); x_k)$$    is  also compact with nonnegative curvature operator.  Then  by the convergence,   $N_j'$ is    diffeomorphic to $\hat N_j$,  i.e.,  a quotient of  sphere.    Thus  again by the classification  of  $(N_j', h_j'(t))$  as in  (\ref{M-split}) for  $( N,  h(t))$,
we know that $(N_j', h_j'(t))$   is a shrinking Ricci soliton on a quotient sphere with  positive curvature operator. By \cite{BS},  it follows that $(N_j', h_j'(t))$ is  a family of  shrinking  round quotient spheres.
   %  have  positively curved curvature.
    %We know that such shrinking soliton must have constant sectional curvature $c>0$.
    As a consequence, the sectional curvature of $(\hat N_j, \hat h_j(\cdot, t_k))$ lies in the interval $[\frac{c_j-\epsilon_k}{-t_k},\frac{c_j+\epsilon_k}{-t_k}]$, where $\epsilon_k\rightarrow 0$ as $k\rightarrow \infty$.
     Hence,  by the curvature pinching estimate in \cite{BS},
      each  $(\hat N_j, \hat h_j(t))$ is  also a  family of shrinking round  quotient  spheres.  As a consequence,    $( N,  h(t))$ is of  type \uppercase\expandafter{\romannumeral1},   which contradicts to the type \uppercase\expandafter{\romannumeral2} condition.  Therefore, $(N', h'(t))$ must be noncompact.
    The lemma is proved.
\end{proof}

% \begin{rem}By

% \end{rem}

 \begin{lem}\label{type2-diameter}
     Let $(N^{n-1},h(t))$ be  an $(n-1)$-dimensional compact  ancient  $\kappa$-solution of  type \uppercase\expandafter{\romannumeral2}.
     % to the Ricci flow with positive curvature operator.
      Then  for any sequence $t_k\rightarrow -\infty$, it holds
     \begin{align*}
         \rm{R}_{min}(t_k)\rm{Diam}(h(t_k))^2\rightarrow \infty,
     \end{align*}
    where ${\rm R}_{min}(t)=\min\{R(h( \cdot, t)\}$.
    Consequently,
     \begin{align}\label{infty-diam}
         \lim_{t\rightarrow -\infty} \rm{R}_{min}(t)\rm{Diam}(h(t))^2\rightarrow \infty.
     \end{align}

 \end{lem}

 \begin{proof}
  By a result of Perelman \cite[Section 11]{P}, for any sequence $t_k\rightarrow -\infty$, we can always find a sequence of points $x_k\in N$ such that $\ell(x_k,t_k)\leq \frac{n-1}{2}$ for each $k$. Then by Lemma \ref{noncompact-asymptotic-soliton}, the rescaled Ricci flows $(N,(-t_k)^{-1}h((-t_k) t); x_k)$ converge to a noncompact shrinking Ricci soliton  $(N', h'(t))$ $(t<0)$ with nonnegative curvature operator.  It follows
 \begin{align}\label{diam-large}
         {\rm Diam}(  (-t_k)^{-1} h(t_k))=
        (-t_k)^{-\frac{1}{2}}{\rm Diam}(h(t_k))
         \to \infty.
     \end{align}
    Moreover,  such a limit soliton   is non-flat \cite[Proposition 39.1]{KL} and so $R(h'(\cdot, t))>0$ for any $t<0$  (cf. \cite{Ch}).   Thus  there exists a uniform constant $\delta>0$, such that $(-t_k) R(x_k,t_k)\geq \delta$ for all $k>>1$.
       Hence, by (\ref{diam-large}), we derive
     \begin{align}\label{Rmax-infty}
         {\rm R}_{max}(t_k){\rm Diam}(h(t_k))^2&\geq {\rm R}(x_k,t_k){\rm Diam}(h(t_k))^2\notag\\
         &\geq \delta (-t_k)^{-1}{\rm Diam}(h(t_k))^2\notag\\
         &\to \infty,
     \end{align}
      where ${\rm R}_{max}(t)=\max\{R(h( \cdot, t)\}$.

 Next  we shall  improve  (\ref{Rmax-infty}) to (\ref{infty-diam}).   Suppose  that   (\ref{infty-diam})  fails.  Then there exist a constant $C$ and a sequence of times $t_k'\to-\infty$,  such that
 \begin{align}\label{diameter-bound-ancient}
     {\rm R}_{min}(t_k'){\rm Diam}(h(t_k'))^2\leq C.
 \end{align}
 Choose $p_k$ such that  ${\rm R}_{min}(t_k')=\min\{R(h( p_k, t_k')\}$.  By  the compactness  of ancient $\kappa$-solutions (cf. \cite[Theorem 3.3]{DZ-TAMS}, \cite[Theorem 20.9]{RF-part3}),  rescaled flows $(N, h_k(t);p_k)$, where $h_k(t)=R(p_k,t_k')h(t_k+R(p_k,t_k')^{-1}t)$, converge subsequently to an ancient $\kappa$-solution $(N_\infty, h_\infty(t); p_\infty)$. By $(\ref{diameter-bound-ancient})$,  we see that $(N_\infty, h_\infty(t); p_\infty)$ is a closed ancient $ \kappa$-solution.  Thus  $R_\infty(\cdot,0)$ is uniformly  bounded on $N_\infty$.

 On the other hand,  by  (\ref{Rmax-infty}) and  (\ref{diameter-bound-ancient}),
  \begin{align}\label{curvature-ratio-infty}
     \frac{R(q_k,t_k')}{R(p_k,t_k')}\to \infty,
 \end{align}
 where  $q_k$ is chosen as  ${\rm R}_{max}(t_k')=\max\{R(h( q_k, t_k')\}$.
 Note
$$\lim_k\frac{R(q_k,t_k')}{R(p_k,t_k')}=R_\infty(q_\infty, 0), $$
 % Moreover, by $(\ref{diameter-bound-ancient})$,
 where $q_\infty\in N_\infty$ is a limit of  $q_k$.  Thus  $R_\infty(q_\infty,0)=\infty$, which is impossible!  The lemma is proved.

\begin{rem}   (\ref{infty-diam})  also comes from   (\ref{Perelman-type2}) and the classification  Theorem \ref{Brendle} when $n=4$.
\end{rem}

  \end{proof}

 \begin{proof}[Proof of  Proposition \ref{compact-limit-ancient-solution}]
     If the proposition is false,  there will  exist a sequence of rescaled flows  $(M,g_{p_i}(t),p_i)$, which  converge   subsequently to a splitting Ricci flow $(N\times\mathbb{R},h(t)+ds^2; p_\infty)$, where $(N, h(t))$ is an  $(n-1)$-dimensional compact  ancient $ \kappa$-solution of type \uppercase\expandafter{\romannumeral2}.
Choose $t_k\to -\infty$. by Lemma \ref{type2-diameter},
    % it is easy to see
    \begin{align}\label{diameter-to-infty}
        \lim_{k\rightarrow\infty}{\rm Diam}(h(t_k))R_h^{\frac{1}{2}}(p_\infty,t_k)\geq\lim_{k\rightarrow\infty}{\rm Diam}(h(t_k))R_{h,min}^{\frac{1}{2}}(t_k)=\infty.
    \end{align}
    Thus   there is  $t_0\in\{t_k\}$ for some $k_0$ such that
        \begin{align}\label{type2-increasing}
            {\rm Diam}(h(t_0))R_h^{\frac{1}{2}}(p_\infty,t_0)>100C.
        \end{align}

Set $T=t_{0}R^{-1}(p_i)$ and choose  $\epsilon<-\frac{1}{100t_{0}}$.
  Then
  $(M,g_{p_i}(t);p_i)$ is $\epsilon$-close to $(N\times \mathbb{R},h(t)+ds^2; p_{\infty})$ when $i>>1$.  Thus by (\ref{type2-increasing}) and the convergence of  $g_{p_i}(t)$,  we get
      \begin{align}\label{qi-sequence-large}
         {\rm Diam}(\bar g(T))R^\frac{1}{2}(p_i,T)= {\rm Diam}(\bar g_{p_i}(t_{0}))R_{p_i}^\frac{1}{2}(p_i,t_{0})
          >90C,
      \end{align}
      as long as  $i>>1$, where $\bar g(T)$ is the induced metric of $g$ on $f^{-1}(f(\phi_T(p_i)))$ and $\bar g_{p_i}(t)$ is the induced metric of $g_{p_i}(t)$ on $f^{-1}(f(p_i))$, respectively.
      It turns
      \begin{align}\label{level-diam}
          {\rm Diam}(R(\phi_T(p_i))\bar g(T))&={\rm Diam}(\bar g(T))R^\frac{1}{2}(\phi_T(p_i))\notag\\
          &={\rm Diam}(\bar g(T))R^\frac{1}{2}(p_i,T)\notag\\
          &>90C,~ i>>1.
      \end{align}

      Let $(\tilde N, \tilde h(t))$ be an $(n-1)$-dimensional split ancient  flow as the limit of  $( M, g_{(\phi_T(p_i))}(t); \phi_T(p_i))$.   Then by  Remark \ref{N-levelset}, $(\tilde N, \tilde h(0))$
     is a limit of submanifolds
     $$( f^{-1}(f(\phi_T(p_i))), R(\phi_T(p_i))\bar g(T); \phi_T(p_i)).$$
       Thus by (\ref{level-diam}), we get
     $${\rm Diam}(\tilde N, \tilde h(0))\ge 90C.$$
     It follows
       $$F(\phi_{T}(p_i))>50C, ~ i>>1. $$
       So
     \begin{align}\label{max-diameter-T1}
         \limsup_{i\rightarrow \infty}F(\phi_{T}(p_i))\geq 50C.
     \end{align}
    However,     by the condition of proposition,    we see
     \begin{align}\label{max-diameter}
         \limsup_{p\rightarrow \infty}F(p)<2C,
     \end{align}
    which contradicts to  $(\ref{max-diameter-T1})$.  Hence, $(N, h(t))$ must be  of type \uppercase\expandafter{\romannumeral2}.
   By   \cite{Ni},   we also get (\ref{N-split}).
The proposition is proved.
 \end{proof}

Now we are able to  prove Theorem \ref{main-theorem} by Proposition \ref{compact-limit-ancient-solution} together with Proposition \ref{noncompact-ancient-solution}.

 \begin{proof}[Proof of Theorem \ref{main-theorem}]
     Case 1:
     $$\limsup_{p\rightarrow \infty}F_\epsilon(p)<C. $$
     for any $\epsilon<1$.  Then  $(\ref{bound-h})$ holds  for all split  ancient $\kappa$-solution $h(t)$.  Thus by  (\ref{N-split}) in case of $n=3$,  all $h(t)$ must be a family of $3d$ shrinking round quotient spheres.

     Case 2:
     $$\limsup_{\epsilon\to 0}\limsup_{p\rightarrow \infty}F_\epsilon(p)=\infty.$$
     In this case, by taking a diagonal subsequence, there is a sequence of pointed flows $(M,g_{q_i}(t);q_i)$, which  converges  subsequently to a splitting Ricci flow $(N'\times\mathbb{R}, h'(t)+ds^2; q_\infty)$ for some noncompact ancient $\kappa$-solution $h'(t)$. Then by Proposition \ref{noncompact-ancient-solution},  $(N, h(t))$ is  also  noncompact from  any splitting  limit flow $(N\times\mathbb{R}, h(t)+ds^2; p_\infty)$ of $(M,g_{p_i}(t);p_i)$.
 \end{proof}

 \begin{proof}[Proof of Corollary \ref{coro}] By the assumption,  the  $3d$ split ancient flow
  $(N,h(t))$ of limit of $(M,g_{p_i}(t),p_i)$ is  a family of shrinking  round quotient spheres. Namely,  $(N,h(0))$ is a quotient  of
 round  sphere. We need to show that  $(M,g)$ has positive Ricci curvature on $M$.

  On the contrary,   ${\rm Ric}(g)$   is not strictly positive.   We note that   $(\ref{pi-decay})$ is still true in the proof of  Lemma \ref{compact-R-decay} without  ${\rm Ric}(g)>0$ away from a compact set of $M$.
   Then as in the proof of \cite[Lemma 4.6]{DZ-JEMS}, we see that  $X_i=R(p_i)^{-\frac{1}{2}}\nabla f\to X_\infty$ w.r.t.  $(M,g_{p_i}(t),p_i)$,  where $X_\infty$ is a non-trivial parallel vector field.  Thus  according to the argument  in  the proof of \cite [Theorem 1.3]{DZ-JEMS},    the universal cover of $(N,h(t)) $ must split off a flat factor $\mathbb{R}^d$ ($d\geq 1$).  However, the universal cover of $N$ is $S^3$.  This is a contradiction!  Hence, we conclude that ${\rm Ric}(g)>0$ on $M$.

  Now we can apply   Theorem \ref{main-theorem} to see that  any  $3d$  split ancient flow  $(N',h'(t))$ of limit of $(M,g_{q_i}(t),q_i)$ is  a family of shrinking  round  quotient spheres.
   We claim:   $(N',h'(t))$ is  in fact a  family of shrinking round spheres.

By Lemma \ref{compact-R-decay},  the scalar curvature of  $(M, g)$ decays to zero uniformly.  Then $(M, g)$ has unique equilibrium point $o$ by the fact  ${\rm Ric}(g)>0$.   Thus  the level set $\Sigma_r=\{f(x)=r\}$  is a closed manifold for any $r>0$, and it is diffeomorphic to $S^3$ (cf. \cite[Lemma 2.1]{DZ-JEMS}).

  On the other hand,   as in the proof of Lemma \ref{compact-ancient-solution},     the level set
     $(\Sigma_{f(q_i)},  \bar g_{q_i};  q_i)$
       converges  subsequently to  $(N', h'(0); q_\infty)$
       w.r.t.  the induced metric  $\bar g_{q_i}$  on $\Sigma_{f(q_i)}$  by  $g_{q_i}$.  Since each $\Sigma_{f(q_i)}$ is diffeomorphic to $S^3$,    $N'$ is also  diffeomorphic to $S^3$. Thus   $(N',h'(t))$ is a  family of shrinking round spheres.

      By the above claim,
   the  condition (ii)  in Definition \ref{asymptotically-cylindrical} is satisfied. Thus  by \cite[Lemma 6.5]{DZ-SCM},  $(M,g)$ is asymptotically cylindrical.  It follows that  $(M,g)$ is isometric to the Bryant Ricci soliton up to scaling by \cite{Bre-high}. Hence, we have shown the first part of   corollary.

   By the above argument, we see that   all $3d$ split limit flow  $(N', h'(t))$ must be a   family of shrinking round spheres,  and so  $(M,g)$ is   isometric to the Bryant Ricci soliton up to scaling, if  ${\rm Ric}(g)>0$ on $M$. Thus the  corollary holds.  On the other hand, if  ${\rm Ric}(g)$ is not strictly positive on $M$, by the splitting  result (cf. \cite{GLX}),  the universal covering of $(M, g)$ must split off a line. Since  $(M, g)$ is $\kappa$-noncollapsed, by  the classification  result of Brendle  for $3d$ $\kappa$-noncollapsed  steady Ricci soliton  \cite{Bre-3d},  $(M, g)$ should be a flat metric on $\mathbb R^4$ or
$M=\tilde N\times \mathbb R$, where  $\tilde N$ is the $3d$ Bryant  Ricci soliton up to scaling. However, in these two cases,  all $3d$ split limit flow  $(N', h'(t))$ by  Proposition \ref{Dimension-reduction}  will also  split out a line.  This  is impossible,  since  there is  a   compact  split  limit flow  $(N, h(t))$ by the assumption!   Hence,  ${\rm Ric}(g)>0$ on $M$ and  the proof  is  complete.

  \end{proof}

Based on  Proposition \ref{compact-limit-ancient-solution} we propose the following conjectures  for general  dimensional case   of Theorem \ref{main-theorem}  and  Corollary \ref{coro}, respectively.

 \begin{conj}\label{high-Th}
  Let $(M^n ,  g)$  ($n\ge 4)$ be an $n$-dimensionl  noncompact $\kappa$-noncollapsed  steady  gradient   Ricci soliton with  $\rm{Rm}\geq 0$   and $\rm{Ric}> 0$ away from a compact set $K$ of $M$.   Let $p_i\rightarrow \infty$ be any  sequence in $M$  and  $ \bar g(t)= h(t) +ds^2$  the   splitting limit flow of    $(M,g_{p_i}(t); p_i)$  as in (\ref{splitt-solution}).  Then either all  $ h(t)$ is an  $(n-1)$-dimensional  compact ancient $\kappa$-solution, which satisfies (\ref{N-split}),   or   each  $ h(t)$  is a  noncompact  ancient $\kappa$-solution of $(n-1)$-dimension.

 \end{conj}

 \begin{conj}\label{high-Coro}

   Let $(M^n, g)$ ($n\ge 4)$  be an   $n$-dimensionl   noncompact $\kappa$-noncollapsed  steady  gradient   Ricci soliton
     with   nonnegative sectional curvature.    Suppose  that there exists a sequence of rescaled flows  $(M,g_{p_i}(t); p_i)$ of  $(M,g)$
which   converges  subsequently to  an  $(n-1)$-dimensianl    compact split limit flow $(N, h(t))$ as in (\ref{splitt-solution}). Then  $(M, g)$ is isometric to  the  Bryant  Ricci soliton of dimension $n$  up to scaling.

 \end{conj}


\newpage

  \begin{thebibliography}{99}
  \bibitem{ABDS} Angenent, S., Brendle,  S., Daskalopoulos, P. and  Sesum, N., \textit{Unique asymptotics of compact ancient solutions to three-dimensional Ricci flow}, \emph{Comm. Pure Appl. Math.}, \textbf{75} (2022), 1032-1073.

  \bibitem{App} Appleton,  A., \textit{A family of non-collapsed steady Ricci solitons in even dimensions greater or equal to four},
  arXiv:1708.00161.

  \bibitem{BK} Bamler, R. and Kleiner, B., \textit{On the rotational symmetry of 3-dimensional $\kappa $-solutions}, \emph{J. Reine Angew. Math.}, \textbf{779} (2021), 37-55.

  \bibitem{Bre-3d} Brendle,  S., \textit{Rotational symmetry of self-similar solutions to the Ricci flow}, \emph{Invent. Math.}, \textbf{194} (2013), 731-764.

  \bibitem{Bre-high} Brendle,  S., \textit{Rotational symmetry of Ricci solitons in higher dimensions}, \emph{J. Differential Geom.}, \textbf{97} (2014), 191-214.

  \bibitem{Bre-3d-noncpt} Brendle,  S., \textit{Ancient solutions to the Ricci flow in dimension 3}, \emph{Acta Math.}, \textbf{225} (2020), 1-102.

  \bibitem{BDS} Brendle,  S., Daskalopoulos, P. and  Sesum, N., \textit{Uniqueness of compact ancient solutions to three-dimensional Ricci flow}, \emph{Invent. Math.}, \textbf{226} (2021), 579-651.

  \bibitem{BS} B\"ohm, C. and Wilking, B., \textit{Manifolds with positive curvature operators are space forms}, \emph{Ann. Math.}, \textbf{167} (2008), 1079-1097.

  \bibitem{BW} Brendle, S. and  Schoen, R., \textit{Manifolds with 1/4-pinched curvature are space forms}, \emph{J. Amer. Math. Soc.}, \textbf{22} (2009), 287-307.

  \bibitem{Bry} Bryant, R., \textit{Ricci flow solitons in dimension three with $ SO (3)$-symmetries}, \emph{preprint, Duke Univ}, (2005), 1-24.

  \bibitem{Ch} Chen, B.L., \textit{Strong uniqueness of the Ricci flow}, \emph{J. Diff.  Geom.},  \textbf{82} (2009),  363-382.
  %\bibitem{Cao2} Cao,  H.  D.  and Chen,  Q.,  \textit{On locally conformally flat gradient steady Ricci solitons}, \emph{Trans. Amer. Math. Soc.}, \textbf{364} (2012), 2377-2391.

  \bibitem{RF-part3} Chow,  B., Chu, S.-C., Glickenstein, D., Guenther, C., Isenberg, J., Ivey, T., Knopf, D.,  Lu,  P., Luo, F.  and Ni,  L.,  \textit{The Ricci flow: techniques and applications. Part III: geometric-analytic aspects}, American Mathematical Soc., 163, 2006.


  \bibitem{CDM} Chow,  B.,  Deng,  Y. and Ma,  Z.,  \textit{On four-dimensional steady gradient Ricci solitons that dimension reduce}, \emph{Adv. Math.}, \textbf{403} (2022), 61 pp.

  \bibitem{CLN} Chow,  B.,  Lu,  P.  and Ni,  L.,  \textit{Hamilton's Ricci flow}, American Mathematical Soc., 2006.



  \bibitem{DW} Dancer,  A.  and Wang,  M., \textit{Some New Examples of Non-K\"ahler Ricci Solitons}, \emph{Math. Res. Lett.}, \textbf{16} (2009), 349-363.

  \bibitem{DZ-TAMS} Deng,  Y. and  Zhu,  X. H.,  \textit{Asymptotic behavior of positively curved steady Ricci solitons}, \emph{Trans. Amer. Math. Soc.}, \textbf{380}  (2018), 2855-2877.

  \bibitem{DZ-JEMS} Deng,  Y. and  Zhu,  X. H.,  \textit{Higher dimensional steady Ricci solitons with linear curvature decay}, \emph{J. Eur. Math. Soc. (JEMS)}, \textbf{22}  (2020), 4097-4120.

  \bibitem{DZ-SCM} Deng,  Y.  and Zhu,  X. H.,  \textit{Classification of gradient steady Ricci solitons with linear curvature decay}, \emph{Sci. China Math.}, \textbf{63} (2020), 135-154.

      \bibitem{Ham82} Hamilton, R.S., \textit{Three manifolds with positive Ricci curvature}, \emph{J. Diff. Geom.}, \textbf{17} (1982), 255-306.

     \bibitem{GLX} Guan, P.F., Lu, P. and Xu, Y.Y., \textit{A rigidity theorem for codimension one shrinking gradient Ricci solitons in $\mathbb{R}^{n+1}$}  \emph{Calc. Var. Partial Differential Equations},  \textbf{54} (2015), no. 4, 4019-4036.


  \bibitem{Ham-eter} Hamilton,  R. S.,  \textit{Eternal solutions to the Ricci flow}, \emph{J. Differential Geom.}, \textbf{38} (1993), 1-11.

  \bibitem{Ham-singular} Hamilton,  R. S., \textit{Formation of singularities in the Ricci flow}, \emph{Surveys in Diff. Geom.}, \textbf{2} (1995), 7-136.

  \bibitem{KL} Kleiner,  B. and Lott, J., \textit{Notes on Perelman's papers}, \emph{Geom. Topol.}, \textbf{12} (2008), 2587-2855.

   \bibitem{Yi-flyingwings} Lai,  Y., \textit{A family of 3d steady gradient solitons that are flying wings}, arXiv:2010.07272, 2020.

  \bibitem{MT} Morgan, J. and Tian, G., \textit{Ricci flow and the Poincare conjecture}, \emph{Clay
Math. Mono., 3. Amer. Math. Soc., Providence, RI; Clay Mathematics
Institute, Cambridge, MA}, 2007, xlii+521 pp. ISBN: 978-0-8218-4328-4.

  \bibitem{MW} Munteanu, O. and  Wang, J. P.,  \textit{Positively curved shrinking Ricci solitons are compact}, \emph{J. Differential Geom.}, \textbf{106} (2017), 499-505.

  \bibitem{Ni} Ni, L., \textit{Closed type I ancient solutions to Ricci flow}, \emph{Recent advances in geometric analysis}, ALM, vol.11 (2009), 147-150.

  \bibitem{P} Perelman,  G.,  \textit{The entropy formula for the Ricci flow and its geometric applications}, arXiv:0211159, 2002.

   \bibitem{P2} Perelman,  G.,  \textit{Ricci flow with surgery on Three-Manifolds}, arXiv:0303109, 2003.

   \bibitem{Zh} Zhang,  Z.  H., \textit{On the completeness of gradient Ricci solitons}, \emph{Proc. Amer. Math. Soc.}, \textbf{137} (2009), 2755-2759.

  \bibitem{ZZ} Zhao,  Z. Y. and Zhu, X. H., \textit{Rigidity of the Bryant Ricci soliton}, arXiv:2212.02889, 2022.
  \end{thebibliography}
  \end{document}